 \documentclass[final,3p,12pt,times]{elsarticle}
\usepackage{amssymb}
\usepackage{amsmath,amsthm,amssymb, mathrsfs,enumerate,wasysym}
 \usepackage{xypic}
\usepackage[all]{xy}
\newcommand{\MM}{\mathbb{M}}
\newcommand{\Q}{\mathbb{Q}}
\newcommand{\Aux}{\tilde{Y}}
\newcommand{\e}{\varepsilon}
\newcommand{\C}{\mathbb{C}}
\newcommand{\Z}{\mathbb{Z}}
\newcommand{\N}{\mathbb{N}}

\newcommand{\links}{\left(\begin{array}{cc}}
\newcommand{\rechts}{\end{array}\right)}
\newcommand{\hidari}{\left(\begin{array}{c}}
\newcommand{\migi}{\end{array}\right)}
\newcommand{\lord}{\Z_p[\Delta][X]}

\newcommand{\CC}{{\mathcal A}}
\newcommand{\LL}{{\mathcal H}}
\newcommand{\llll}{{h}}
\newcommand{\bb}{{\beta}}

\newcommand{\Spec}{{\mathrm{Spec}}}
\newcommand{\Gal}{{\mathrm{Gal}}}

\newcommand{\cor}{{\mathrm{cor}}}

\newcommand{\Hom}{{\mathrm{Hom}}}
\newcommand{\Tr}{{\mathrm{Tr}}}
\newcommand{\col}{{\mathrm{Col}}}

\newcommand{\rank}{{\mathrm{rank}}}
\newcommand{\coker}{{\mathrm{Coker}\,}}
\renewcommand{\ker}{{\mathrm{Ker}\,}}

\newcommand{\Char}{{\mathrm{Char}}}
\newcommand{\GL}{{\mathrm{GL}}}

\newcommand{\im}{{\mathrm{Im}\,}}
\newcommand{\Sel}{{\mathrm{Sel}}}

\newcommand{\gp}{{\mathfrak p}}

\newcommand{\gm}{{\mathfrak m}}

\renewcommand{\H}{{\mathcal{H}}}
\renewcommand{\qedsymbol}{{$\mathit{QED}$}}
\def\proj{\mathop{\mathrm{proj}}\nolimits}
\def\Iw{\mathop{\mathrm{Iw}}\nolimits}
\renewcommand{\phi}{{\varphi}}

\newcommand{\F}{{\mathcal F}_{ss}}
\newcommand{\G}{{\mathcal G}}
\newcommand{\X}{{\mathcal X}}
\newcommand{\dn}{{\delta_n^i}}
\newcommand{\dI}{{\delta_n^{i+1}}}

\newtheorem{theorem}{Theorem}[section]

\newtheorem{conjecture}[theorem]{Conjecture}
\newtheorem{mainconjecture}[theorem]{Main Conjecture}
\newtheorem{corollary}[theorem]{Corollary}

\newtheorem{example}[theorem]{Example}
\newtheorem{lemma}[theorem]{Lemma}
\newtheorem{propn}[theorem]{Proposition}
\newtheorem{open problem}[theorem]{Open Problem}

\theoremstyle{definition}

\newtheorem{observation}[theorem]{Observation}
\newtheorem{definition}[theorem]{Definition}

\newtheorem{notation}[theorem]{Notation}

\newtheorem{remark}[theorem]{Remark}

\newtheorem{maintheorem}[theorem]{Main Theorem}

\journal{a journal}

\begin{document}

\begin{frontmatter}



\title{Iwasawa theory for elliptic curves at supersingular primes: A pair of main conjectures}


\author{Florian Sprung $\lfloor$``ian"$\rfloor$}

\address{151 Thayer Street, Box \#1917, Providence, RI 02912}

\begin{abstract}
 We extend Kobayashi's formulation of Iwasawa theory for elliptic curves at supersingular primes to include the case $a_p \neq 0$, where $a_p$ is the trace of Frobenius. To do this, we algebraically construct $p$-adic $L$-functions $L_p^{\sharp}$ and $L_p^{\flat}$ with the good growth properties of the classical Pollack $p$-adic $L$-functions that in fact match them exactly when $a_p=0$ and $p$ is odd. We then generalize Kobayashi's methods to define two Selmer groups $\Sel^{\sharp}$ and $\Sel^{\flat}$ and formulate a main conjecture, stating that each characteristic ideal of the duals of these Selmer groups is generated by our $p$-adic $L$-functions $L_p^{\sharp}$ and $L_p^{\flat}$. We then use results by Kato to prove a divisibility statement.
\end{abstract}


\end{frontmatter}


\begin{section}{Introduction}
In the early 1970's, B. Mazur and P. Swinnerton-Dyer constructed a $p$-adic $L$-function for an elliptic curve $E/\Q$ when $p$ is good ordinary. Their $L$-function is an Iwasawa function, i.e. a $p$-adic analytic function convergent on the closed $p$-adic unit disc, and interpolates special values of the complex $L$-series of $E$ twisted by various characters. Later that decade, Mazur formulated the Iwasawa theory for elliptic curves at these good ordinary primes, relating the $p$-adic $L$-function to the Selmer group over the cyclotomic $\Z_p$-extension $\Q_\infty$. The good supersingular case (the one for which the trace of Frobenius $a_p$ is divisible by $p$) has not been in as good a shape yet. Thanks to the work of R. Pollack from the early 2000's, we now have a \textit{pair} of Iwasawa functions in the case $a_p=0$, for which S. Kobayshi was able to formulate a pair of Iwasawa main conjectures by relating each of Pollack's Iwasawa functions to a modified Selmer group. However, the general supersingular case has so far seemed less amenable to analysis. The main theorems of this article obtain an appropriate pair of Iwasawa functions and relate them to a new pair of modified Selmer groups via a pair of main conjectures. Our pairs of Iwasawa functions and Selmer groups compare favorably with the works of Pollack and Kobayashi when reduced to the case $a_p=0$. 

More precisely, let $\Gamma := \Gal(\Q_\infty/\Q)$. We can identify the ring of power series $\Z_p[[X]]$ with $\Z_p[[\Gamma]]$. In the good ordinary case, Mazur's and Swinnerton-Dyer's $p$-adic $L$-function (see \cite{mazurSD}, but also \cite{mazurbourbaki}) is an element of $\Q_p\otimes\Z_p[[X]]$ and thus an Iwasawa function, but conjecturally should live in the simpler ring $\Z_p[[X]]$. Mazur conjectured in \cite{mazurmainconjecture} that the Selmer group is $\Z_p[[X]]$-cotorsion (i.e. the Pontryagin dual is $\Z_p[[X]]$-torsion - Kato proved this in \cite{kato}), and conjectured that the $p$-adic $L$-function generates its characteristic ideal. This is the main conjecture, a proof of which has been announced by Skinner and Urban \cite{skinnerurban}.

In the 1980's, Mazur, Tate, and Teitelbaum constructed more general $p$-adic $L$-functions in \cite{MTT}, among others for the supersingular (a.k.a. extraordinary) case, reconstructing $p$-adic $L$-functions of Vi\v{s}ik \cite{vishik}, and Amice and V\'{e}lu \cite{amicevelu}. This case contains infinitely many primes $p$ by a theorem of Elkies \cite{elkies}, but the $p$-adic $L$-functions are no longer elements of $\Z_p[[X]]\otimes \overline{\Q_p}$, and correspondingly, the Selmer group is no longer $\Z_p[[X]]$-cotorsion. In fact, there are two $p$-adic $L$-functions, one for each root $\alpha$ of the Hecke polynomial $Y^2-a_pY+p$, denoted $L_p(E,\alpha,X)$.

In the last decade, Pollack noticed in \cite{pollacksthesis} that if $a_p=0$, one may get far by constructing auxiliary functions $\log_p^\pm$ vanishing at all $p^{n}$th roots of unity for $n$ even ($+$) or $n$ odd ($-$). Using \textit{analytic} methods, he expressed each of these two $p$-adic $L$-functions as a certain sum, e.g. for $p$ odd,
\[ L_p(E,\alpha,X)=L_p^+(E,X)\log_p^+(1+X)+L_p^-(E,X)\log_p^-(1+X)\alpha,\]
where $L_p^\pm \in \Z_p[[X]]$. Kobayashi then constructed two $\Z_p[[X]]$-cotorsion subgroups of the Selmer group $\Sel^\pm$ in \cite{kobayashisthesis} and formulated a pair of Iwasawa main conjectures in the spirit of Mazur's ordinary one: Each of Pollack's $L_p^\pm$-functions should generate the characteristic ideal of the Pontryagin duals of $\Sel^\pm$. He used \textit{algebraic} arguments to reconstruct Pollack's $L_p^\pm$-functions and to define his Selmer groups, and proved a divisibility statement of his main conjecture via work of Kato, but still assumed $a_p=0$. 

What makes the case $a_p = 0$ so much more convenient than the case $a_p \neq 0$ is that it allowed Pollack and Kobayashi to give \textit{separate} constructions for their appropriate objects $L_p^\pm, \log^\pm, $ and $\Sel^\pm$, which break down for $a_p\neq0$. For example, Pollack worked his observation that the sum $L_p(E,\alpha,X) + L_p(E,\overline{\alpha},X)$ vanishes at $X=\zeta_{p^n}-1$ for even $n$ into the definition of $\log^+$. What makes Pollack's observation work is that $\alpha^2=-p$ is an easy integer, since it is $1$ modulo powers of $p$ and signs! This \textit{two-periodicity} of $\alpha$ and related objects was the crucial ingredient that made the arguments in Kobayashi's paper work as well. In this paper, we give a method that is independent of any periodicity, constructing the essential objects \textit{simultaneously}. For example, there are four appropriate generalizations of $\log^\pm$ which we define as entries of a limit of an infinite product of $2\times2$ matrices. In general, we can't know just one of the four entries without knowing the other three since its definition involves information coming from every factor in the matrix product.
%are cwas a function which had the same zeros as only uncover $\frac{1}{p}$th of the desired zeroes. The sum $L_p(E,\alpha,X)+L_p(E,\overline{\alpha},X)$ only vanishes at $X=\zeta_{p^n}-1$ with $n\equiv0(CHECK IF THIS IS THE RIGHT CONGRUENCE) \mod 2p$ for $a_p\neq0$, while it vanishes at every other $p$-power (that is, when $n\equiv 0 \mod 2$) when $a_p=0$ (cf. \cite[Corollary 3.6]{pollacksthesis})(ACTUALLY COME UP WITH A BETTER EXPLANATION). In this paper, we find the other $\frac{p-1}{p}$th of the zeroes as the zeroes of \textit{linear combinations} of certain intermediary functions, rather than the zeroes of those intermediary functions themselves (which used to be enough when $a_p=0$, and led to \textit{separate} constructions of the $L_p^{\pm}$). For $a_p \neq 0$, the interplay of the intermediary functions lets us construct two new $p$-adic $L$-functions \textit{simultaneously}. 

We then formulate a main conjecture in the spirit of Mazur's that relates our $p$-adic $L$-functions to subgroups of the Selmer group when $p$ is odd, and prove a divisibility statement.

We note that the Hasse-Weil bound $|a_p|\leq 2\sqrt{p}$ (see e.g. \cite[Chapter 5]{silverman}) forces $p=2$ or $p=3$ when $a_p\neq 0$. But we have worked out our methods so that they should generalize to general modular forms of weight two (with larger relevant primes). For elliptic curves, appropriate examples include any curve of conductor $11$ ($a_2=-2$) and a classical curve of Mordell given by $y^2+y=x^3-x$ ($a_2=-2, a_3=-3$). A look at the Cremona tables reveals that $70$ out of the $471$ curves of conductor $\leq 299$ have $a_p\neq 0$ for some supersingular $p$. The first curve with $a_2=2$ is $67A$, that with $a_3=3$ is $140B$, and the first curve with $a_2=2, a_3=3$ is $319A$. The two curves $245A$ and $245B$ both have $a_2=-2$, while $a_3=-3$ for the first, and $a_3=3$ for the latter.

There has been much work in the direction of extending the theory presented in this paper. Antonio Lei \cite{Lei} has generalized Kobayashi's methods to modular forms of higher weight, assuming $a_p=0$. Lei, Loeffler, and Zerbes have given one generalization to any odd prime of good reduction in \cite{llz} for modular forms of higher weight using the theory of Wach modules. They show the existence of pairs of $p$-adic $L$-functions which live in (a ring analogous to\footnote{Instead of $\Z_p$, they work with the ring of integers of the completion at $p$ of the number field generated by the eigenvalues of the Hecke operators and the values of the nebencharacter.}) $\Z_p[[X]]\otimes \Q$, and a matrix that relates them to the $p$-adic $L$-functions of Amice and V\'{e}lu, and Vi\v{s}ik. It would be nice to see if these matrices could be expressed \textit{explicitly}, and to construct $p$-adic $L$-functions in (the ring analogous to) $\Z_p[[X]]$, as in this paper. Their paper suggests that this is possible and is thus a huge hint for this question!

$\text{\bf{Overview. }}$ 
We now give a sketch of our methods in the case of an odd prime $p$. Denote by $\Phi_n(X)=\displaystyle \sum_{t=0}^{p-1} X^{p^{n-1}t}$ the $p^n$th cyclotomic polynomial, which is the irreducible polynomial for any primitive $p^n$th root of unity $\zeta_{p^n}$. %\in \MU_{p^n}$.
%Let $N=n+1$ if $p$ is odd, and $N=n+2$ if $p=2$. 
Put $k_n=\Q_p(\zeta_{p^{n+1}})$, and let $\G_n=\Gal(\Q(\zeta_{p^{n+1}})/\Q)\cong\Gal(k_n/\Q)$. Let $\Lambda_n=\Z_p[\G_n]$. %We then have a decomposition $\G_n\cong (\Z/{p^N\Z})^\times \cong\Delta \times \Gamma_n$, where $\Gamma_n\cong \Z/p^n\Z,$ and $\Delta\cong \Z/(p-1)\Z$.
% if $p$ is odd. For $p=2, \Delta=\{\pm1\}\subset \G_n$.
 Taking inverse limits, put $\G_\infty= \displaystyle\varprojlim_n \G_n$, and let $\Lambda=\Z_p[[\G_\infty]]$. %\cong \Delta \times \Gamma$, where $\Gamma\cong \Z_p$. Now fix a topological generator $\gamma$ of $\Gamma$. By sending $\gamma$ to $(1+X)$, we can identify $\Lambda=\Z_p[[\G_\infty]]$ with $\Z_p[\Delta][[X]]$.
 % Denoting the image of $\gamma$ under the projection $\G_\infty \rightarrow \G_n$ by $\gamma_n$, we can similarly identify $\Lambda_n=\Z_p[\G_n]$ with $\Z_p[\Delta][[X]]/(\omega_n(X))$, where $\omega_n(X)= (1+X)^{p^n}-1$. See \cite[Chapter 7]{washington}.

The key construction in this paper is that of two Coleman maps $\col^{\sharp}$ and $\col^{\flat}$. They are generalizations of Kobayashi's $\col^\pm$ (since for $a_p=0$, $p$ odd, we have $\col^{\sharp}=\col^-, \col^{\flat}=\col^+$). As such, they are $\Lambda$-valued and just as (the trivial character-components of) $\col^\pm$ sent Kato's zeta element to Pollack's $p$-adic $L$-functions, (the trivial character-components of) $\col^{\sharp}$ and $\col^{\flat}$ send Kato's zeta elements to new $p$-adic $L$-functions $L_p^{\sharp}(E,X)$ and $L_p^{\flat}(E,X)$, which are natural generalizations of Pollack's $p$-adic $L$-functions: They are in the trivial character-component $\Z_p[[X]]$ of $\Lambda$ and satisfy the following relation with the classical $p$-adic $L$-function $L_p(E,\alpha,X)$:\begin{theorem}Let $L_p^\sharp(E,X) \in\Z_p[[X]], L_p^\flat(E,X) \in\Z_p[[X]]$, and $L_p(E,\alpha,X)$ be as above. Then we have $L_p(E,\alpha,X)=L_p^{\sharp}(E,X)\log_\alpha^{\sharp}(1+X)+L_p^{\flat}(E,X)\log_\alpha^{\flat}(1+X).$
\end{theorem}
Here, $\log_\alpha^{\sharp}$ and $\log_\alpha^{\flat}$ are generalizations of Pollack's half-logarithms $\log_p^\pm$, which we construct together with the Coleman maps $\col^{\sharp}$ and $\col^{\flat}$ by looking at Kurihara's $P_n$-pairing (see \cite{kurihara}): Denote by $T$ the $p$-adic Tate module. From Honda's theory of formal groups, there is a system of elements $(c_n)_n, c_n\in E(k_n)\otimes \Z_p \subseteq H^1(k_n,T)$ satisfying $\Tr_{n+1/n}(c_{n+1})=a_pc_n-c_{n-1}$. By using the essential assumption $a_p=0$, Kobayashi proved that the image of $P_{n,c_{n-1}}: H^1(k_n,T)\rightarrow \Lambda_n$ vanished at all $X=\zeta_{p^m}-1$ for $m\leq n$ with the same parity as $n$. In this way, he recovered the zeros of $\log_p^\pm$ (the sign depends on the parity of $n$). Dividing the image of $P_{n,c_{n-1}}$ by the minimal polynomial of these zeroes, he constructed maps  $\col_n^\pm$ whose inverse limits would be $\col^\pm$.

To treat the excluded case $a_p\neq0$, these methods do not have to be modified quite that significantly. In Section \ref{generatorsection}, we define certain linear combinations $\delta_n^i$ of $c_n$ and $c_{n-1}$. These $\delta_n^i$ are periodic with respect to $i$, and if one ignores the sign, the periods are $2p$ for $a_p\neq0$ and $2$ for $a_p=0$ (for $a_p=0$, we just have $\delta_n^i=\pm c_n$ or $=\pm c_{n-1}$). Kobayashi's key accomplishment was to exhibit new zeroes (the primitive $p^n$th roots of unity) for $P_{n,c_{n-1}}$ each time $n$ was increased by $2$. If we followed his method na\"{i}vely for $a_p \neq 0$, the periodicity of $\delta_n^i$ would only allow us to discover new zeroes each time $n$ increased by $2p$, which is too few, since the image would then not have the right growth properties. \textit{Rather than relying on any periodicity}, we let linear algebra come to the rescue. We look at certain linear combinations of $P_{n,\delta_n^{i+1}}$ and $P_{n,\delta_n^{i}}$ and show their images vanish at $X=\zeta_{p^n}-1$. In this way, we are able to write the function $(P_{n,c_n}, P_{n,c_{n-1}}): H^1(k_n,T)\rightarrow \Lambda_n^{\oplus 2}$ as a product of a function $\col_n$ times a product $\LL_n$ involving $n$ matrices of the form $\links a_p& \Phi_m \\ -1 & 0\rechts$. Then we prove that $(\col^{\sharp}, \col^{\flat})=\displaystyle\varprojlim_n \col_n$ is a well-defined map from $\mathbf{H}^1_{\Iw}=\displaystyle\varprojlim_n H^1(k_n,T)$ to $\Lambda^{\oplus 2}$.

The outline of the paper is as follows: In Section \ref{generatorsection}, we construct the generators $\delta_n^i$ and prove some basic properties. In Section \ref{constructionofthetaupsilon}, we construct the product of matrices $\LL_n$ that we just mentioned, and relate it to the pair of functions  $(P_{n,c_n}, P_{n,c_{n-1}})$. Then in Section \ref{growthproperties}, we construct a limit matrix $\LL$ and prove that its entries are in $O(\log_p(1+X)^{\frac{1}{2}})$, as are Pollack's half-logarithms $\log_p^\pm(1+X)$. In Section \ref{colemanmaps}, we construct the Coleman maps $\col^{\sharp}$ and $\col^{\flat}$ and in Section \ref{padiclfunctions}, we construct $\log_\alpha^{\sharp}$ and $\log_\alpha^{\flat}$ from $\LL$ and prove our main theorem. 

In Section \ref{mainconjecture}, we give a formulation of a pair of main conjectures for \textit{any} odd supersingular prime, each in the spirit of Mazur's. Each is equivalent Kato's \cite{kato}, Perrin-Riou's \cite{perrinriou}, or Kurihara's \cite{kurihara} main conjecture. Let $\gp$ be the prime ideal of $\Q_\infty$ above $p$, and
\[\Sel^{\sharp}(E/\Q_\infty):=\ker\left(\Sel(E/\Q_\infty)\longrightarrow\frac{E(\Q_{\infty,\gp})\otimes\Q_p/\Z_p}{E^{\sharp}_{\infty,\gp}}\right),\]
where the local condition $E_{\infty,\gp}^{\sharp}$ is the exact annihilator under the local Tate pairing of $\ker \col^{\sharp}$, and similarly define $\Sel^{\flat}(E/\Q_\infty)$. The idea of choosing a new local condition to get amenable modified Selmer groups goes back to Kurihara, who chose the trivial group for his main conjecture, but did not have an Iwasawa function yet for the analytic side. Now at least one of $L_p^{\flat}(E,X)$ and $L_p^{\sharp}(E,X)$ is nonzero, and conjecturally both are. The theorem below is the key to our main conjecture:\begin{theorem} Choose $*\in\{\sharp,\flat\}$ so that $L_p^*(E,X)$ is nonzero. The Pontryagin dual $\X^*(E/\Q_\infty)=\Hom(\Sel^*(E/\Q_\infty),\Q_p/\Z_p)$ of this $*$-Selmer group is a finitely generated torsion $\Z_p[[X]]$-module. \end{theorem}
\begin{mainconjecture}
Let $p$ be odd, and $*\in\{\sharp,\flat\}$ so that $L_p^*(E,X)$ is nonzero. The characteristic ideal of the Pontryagin dual of $\Sel^*(E/\Q_\infty)$ is then generated by the $p$-adic $L$-function $L_p^*(E,X)$:
\[\Char(\X^*(E/\Q_\infty))=(L_p^*(E,X)).\]
\end{mainconjecture}

We then use a theorem of Kato concerning his Euler systems to prove the following theorem:

\begin{theorem} Suppose $E/\Q$ does not have complex multiplication. Choose $*\in\{\sharp,\flat\}$ so that $L_p^*(E,X)$ is nonzero. Then if the $p$-adic representation $\Gal(\overline{\Q}/\Q)\rightarrow \GL_{\Z_p}(T)$ on the automorphism group of the $p$-adic Tate module $T$ is surjective, we have 
\[\Char(\X^*(E/\Q_\infty))\supseteq(L_p^*(E,X)).\]
\end{theorem}

%We can prove that both $L_p^{\sharp}(E,X)$ and $L_p^{\flat}(E,X)$ are nonzero in certain cases (e.g. if $a_p=0$, or if the rank of $E$ is zero and $a_p\neq 3$), and thus make the following
%\begin{conjecture}
%Both $L_p^{\sharp}(E,X)$ and $L_p^{\flat}(E,X)$ are nonzero.
%\end{conjecture}
The statements here correspond to the $\eta=1$ case in the more precise statements of Section \ref{mainconjecture}. In fact, we will work with the Iwasawa algebra $\Z_p[\Delta][[X]]$ rather than $\Z_p[[X]]$. The smart reader will notice that some of our methods can (almost word-for-word) be applied to the setting of \cite{iovitapollack}, where the authors work with a number field for which $p$ (assumed to be odd) splits completely. Since the applications of the results there are however not generalizable to the case $a_p\neq0$ at present and we included the prime $p=2$ in our paper, we have decided to stick with Kobayashi's setting for convenience of the reader and simplicity. 

\end{section}

\begin{section}{The generators $\delta_n^i$}\label{generatorsection}$\text{\bf{Notation. }}$ 
We attempt to keep the notation of \cite{kobayashisthesis}. Thus, $E$ is an elliptic curve over $\Q$ and $p$ a prime so that $E$ has good supersingular reduction at $p$, $[\frac{a}{b}]$ is the greatest integer not greater than $\frac{a}{b}$, $I$ the identity matrix, and $\Phi_n(X)=\displaystyle \sum_{t=0}^{p-1} X^{p^{n-1}t}$ the $p^n$th cyclotomic polynomial, which is the irreducible polynomial for any primitive $p^n$th root of unity $\zeta_{p^n}$. %\in \MU_{p^n}$.
Let $N=n+1$ if $p$ is odd, and $N=n+2$ if $p=2$. Put $k_n=\Q_p(\zeta_{p^N})$, and let $\G_n=\Gal(\Q(\zeta_{p^N})/\Q)\cong\Gal(k_n/\Q)$. We then have a decomposition $\G_n\cong (\Z/{p^N\Z})^\times \cong\Delta \times \Gamma_n$, where $\Gamma_n\cong \Z/p^n\Z,$ and $\Delta\cong \Z/(p-1)\Z$ if $p$ is odd. For $p=2, \Delta=\{\pm1\}\subset \G_n$.
 Taking inverse limits, put $\G_\infty= \displaystyle\varprojlim_n \G_n \cong \Delta \times \Gamma$, where $\Gamma\cong \Z_p$. Now fix a topological generator $\gamma$ of $\Gamma$. By sending $\gamma$ to $(1+X)$, we can identify $\Lambda=\Z_p[[\G_\infty]]$ with $\Z_p[\Delta][[X]]$.
 Denoting the image of $\gamma$ under the projection $\G_\infty \rightarrow \G_n$ by $\gamma_n$, we can similarly identify $\Lambda_n=\Z_p[\G_n]$ with $\Z_p[\Delta][[X]]/(\omega_n(X))$, where $\omega_n(X)= (1+X)^{p^n}-1$. See \cite[Chapter 7]{washington}.
 We denote by $\gm_n$ the maximal ideal of $\Z_p[\zeta_{p^N}]$.
 \begin{definition}\[A:=\links a_p & p \\ -1 & 0 \rechts.\]
 \end{definition}

Kobayashi constructed a power series $\log_{\F}(X) \in \Z_p[[X]]$ that can be interpreted as the logarithm of a formal group $\F$ isomorphic to the formal group $\hat{E}$ of the elliptic curve via Honda theory \cite{honda}. Although Kobayashi assumed $a_p=0$, Pollack has pointed out  in \cite{pollack} that this could be done for $a_p \neq 0$ as well. Both \cite{kobayashisthesis} and \cite{pollack} also assume $p$ is odd, but Honda's theorems hold for $p=2$ as well \cite{honda}, so we make no special assumption on $a_p$ or $p$ in this paper until Section \ref{mainconjecture}. We denote the trace map from $\F(\gm_{n+1})$ to $\F(\gm_n)$ by $\Tr_{n+1/n}$. The power series $\log_{\F}$ can be used to show the following result of Kobayashi:

\theorem\label{tracerelations} There exist $c_n\in\F(\gm_n)\cong\hat{E}(\gm_n)$ so that as a $\Z_p[\G_n]$-module, $\F(\gm_n)$ is generated by $c_n$ and $c_{n-1}$ when $n\geq 0$. $\F(\gm_{-1})$ is generated by $c_{-1}$ when $p$ is odd, and $\F(\gm_{-2})=\F(\gm_{-1})$ by $c_{-2}$ when $p=2$. They satisfy the relations:
\begin{description}
\item{(1)} $\Tr_{n+1/n} c_{n+1}=a_pc_n-c_{n-1} \text{ if $n\geq0$},$
\item{(2)} $\Tr_{0/{-1}}c_0=(a_p-2)c_{-1} \text { when $p$ is odd,}$
\item{(2')} $\Tr_{0/{-1}}c_0=(\frac{a_p}{p}-1)c_{-1}+2c_{-2} \text { when $p=2$.}$
\end{description}
\begin{proof}
This is essentially \cite[Proposition 8.12]{kobayashisthesis} and \cite[Lemma 8.9]{kobayashisthesis}. See also the discussion in \cite[Theorem 3.1]{pollack} for the cyclotomic $\Z_p$-extension of $\Q_p$ with odd $p$. To allow arbitrary supersingular primes, we must modify Kobayashi's arguments, but not too much: What we need is a formal group whose logarithm is of Honda type $t^2-a_pt+p$. We thus look at the sequence $\{x_k\}$ given by
%by setting $x_{-1}=0,x_0=1$, and inductively requiring $px_k-a_px_{k-1}+x_{k-2}=0$, or more explicitly putting
\[(x_k,x_{k-1}):=(1,0)A^k\times\frac{1}{p^k} \text{ for $k\geq 0$}.\] 
The logarithm giving rise to our formal group via Honda theory (\cite[Theorem 8.3 iii)]{kobayashisthesis}) is then the power series 
\[\log_{\F}(X)=\sum_{k=0}^\infty x_k((1+X)^{p^k}-1).\]
%We can do this since the logarithm of any formal group over $\Z_p$ is bijective on $p\Z_p$ if $p$ is odd, and $\log_{\F}$ is bijective on $2\Z_2$ for $p=2$ as well (cf. Lemma \ref{noptorsion} below). Now put $c_n = \e[+]_{\F}(\zeta_{p^N}-1)$. Then $c_n, c_{n-1}$ are generators of $\F(\gm_n)$ as a $\Z_p[\G_n]$-module for $n\geq 0$. (The ideal $\F(\gm_{-1})$ is generated by $c_{-1}$ if $p$ is odd, but not for $p=2$: e.g. if $a_2=2$, then $c_{-1}=0$.)
Note that to make Kobayshi's arguments work, we choose $\e \in p\Z$ so that log$_{\F}(\e)=\frac{p}{p+1-a_p}$. Addition in our formal group then allows us to construct $c_n = \e[+]_{\F}(\zeta_{p^N}-1)$.
\end{proof}
\begin{lemma}\label{noptorsion}
$\F(k_n)$ has no $p$-torsion.
\end{lemma}
\begin{proof}
Multiplication by $p$ is given by a power series $[p](X)=pX+\textit{(higher order terms)}{\in\Z_p[[X]]}$ on the formal group . We have $[p](X)=f(X)u(X)$ for a distinguished polynomial $f(X)$ of degree $p^2$ and a unit $u(X)$ by the $p$-adic Weierstra{\ss} preparation theorem. Now $\frac{f(X)}{X}$ has constant term $p\times \textit{(unit)}$, and is thus an Eisenstein polynomial of degree $p^2-1$. It follows that the nontrivial torsion points live in a totally ramified extension of degree $p^2-1$, which does not divide $\deg(k_n/\Q_p)$: For odd $p$, this degree is $p^{n+1}-p^n$, while for $p=2$, we have $\deg(k_n/\Q_p)=p^{n+1}$. Thus the only $p$-torsion point in $\F(k_n)$ is zero itself.
\end{proof}
\rm For the case $a_p=0$, Kobayashi used a certain trace-compatibility of these generators \cite[Lemma 8.9]{kobayashisthesis} to define elements $c_n^\pm\in \F(\gm_n)$, from which he was able to reconstruct Pollack's $L_p^\pm$-functions. We will now give a more general construction that includes the case  $a_p\neq0$. 

\begin{definition}\label{bestdefinition}
Let $i\in\Z$. Put
\[B_i:=
B_+=\links \frac{1}{p} & 0 \\ 0 & 1 \rechts \text{ if $i$ is even, and }  B_i:=
B_-=\links1 & 0 \\ 0 &  \frac{1}{p} \rechts \text{ if $i$ is odd, and let } Y_0:=I. 
\]
For general $i\in\Z$, we define $Y_i$ by putting $AY_{i-1}B_i=Y_i$ and using two-way induction. Equivalently, $Y_i:=A^iB_1...B_i$ for $i>0$ and $Y_i:=A^iB_{0}^{-1}B_{-1}^{-1}...B_{i+1}^{-1}$ for $i<0.$ We prove below (Integrality Lemma \ref{matrixcalculation}) that $Y_i$ has integral entries. Now put 
\[(\delta_n^{i+1},\delta_n^i):=(c_n,c_{n-1})Y_i \text{ for any $i\in\Z$}.\]
%(\delta_{n+1}^i, \delta_{n+1}^{i-1})=(c_{n+1},c_n)Y_{i-1}

\end{definition}
For the reader's convenience, we include a table of the $\delta_n^i$s. \[
\begin{tabular}{|c||r@{}c@{}r|r@{}c@{}r|r@{}c@{}l|r@{}c@{}l|r@{}c@{}r|}
\hline
\multicolumn{1}{|c||}{}& \multicolumn{3}{c|}{$a_p=2$}& \multicolumn{3}{c|}{$a_p=-2$} &\multicolumn{3}{c|} {$a_p=3$} & \multicolumn{3}{c|}{$a_p=-3$} & \multicolumn{3}{c|} {$a_p=0$} 
\\ 
\hline
$\delta_n^{-1}$ & $-c_n$   &$+$&$c_{n-1}$ & $-c_n$&$-$&$c_{n-1}$  & $-c_n$&$+$&$c_{n-1}$   & $-c_n$&$-$&$c_{n-1}$& $-c_n$&&\\ \hline
$\delta_n^0$ &                &$$&$c_{n-1}$ &          &$$&$c_{n-1}$     & &&$c_{n-1}$& &&$c_{n-1}$&&&$c_{n-1}$\\ \hline
$\delta_n^1$ & $c_n$   &      &                   & $c_n$   &&                        & $c_n$&&                             & $c_n$&&                          & $c_n$&&\\ \hline
$\delta_n^2$ & $2c_n$ &$-$&$c_{n-1}$ & $-2c_n$&$-$&$c_{n-1}$& $3c_n$&$-$&$c_{n-1}$  & $-3c_n$&$-$&$c_{n-1}$& &$-$&$c_{n-1}$\\ \hline
$\delta_n^3$ & $c_n$   &$-$&$c_{n-1}$ & $c_n$&$+$&$c_{n-1}$  & $2c_n$&$-$&$c_{n-1}$   & $2c_n$&$+$&$c_{n-1}$& $-c_n$&&\\ \hline
$\delta_n^4$ &                &$-$&$c_{n-1}$ &          &$-$&$c_{n-1}$     & $3c_n$&$-$&$2c_{n-1}$& $-3c_n$&$-$&$2c_{n-1}$&&&$c_{n-1}$\\ \hline
$\delta_n^5$ & $-c_n$  &       &                  & $-c_n$&       &                  & $c_n$&$-$&$c_{n-1}$    & $c_n$&$+$&$c_{n-1}$   & $c_n$&&\\ \hline
$\delta_n^6$ & $-2c_n$&$+$&$c_{n-1}$& $2c_n$&$+$&$c_{n-1}$&           & $-$&$c_{n-1}$    & &$-$&$c_{n-1}$                & &$-$&$c_{n-1}$\\ \hline
$\delta_n^7$ & $-c_n$&$+$&$c_{n-1}$  & $-c_n$&$-$  &$c_{n-1}$& $-c_n$&&                          & $-c_n$&&                         & $-c_n$&&\\ \hline
$\delta_n^8$ &              &     &$c_{n-1}$  &               &        &$c_{n-1}$& $-3c_n$&$+$&$c_{n-1}$& $3c_n$&$+$&$c_{n-1}$ & &&$c_{n-1}$\\ \hline
\end{tabular}
\]
\remark 
When $a_p=0$, the $\delta_n^i$ are periodic in $i$ with period $4$, or $2$ if one ignores signs, which was essential in the work of \cite{kobayashisthesis}. The period is $4p$ (or $2p$ modulo signs) for $a_p\neq 0$, but in this paper, no argument needs this periodicity.

\begin{propn}Let $n\geq0$ and $i\in\Z$. Then
%\begin{itemize}
%\item $\delta_n^0=c_{n-1}.$ Inductively for $i \geq0,$
\[\dI = \quad\begin{cases}\quad \Tr_{n+1/n}(\delta_{n+1}^i)\text{ if $i $ is odd, and }\\
 \quad \frac{1}{p}\Tr_{n+1/n}(\delta_{n+1}^i)\text{ if $i $ is even.}\end{cases} \]
%\end{itemize}
\end{propn}

Note that the sums of the upper and lower indices of $\delta_n^{i+1}$ and $\delta_{n+1}^i$ are the same. %This sum can thus detect \textit{compatible} (i.e. $p$-power trace compatible) systems of $\delta_n^i$. 
%That we could divide by $p$ in the above equation is not clear a priori, and will be justified in what follows. It can also be verified by calculation, best done after reading Section $3$ of \cite{pollack} (It's only one page. There, $\Q_{n,p}:=k_{n-1}$ and the subindices of $c_n$ are shifted by one).  
\begin{proof}
%Now let $n\geq0 $ ($n\geq -1$ if $p=2$). Then 
From theorem \ref{tracerelations}, we have $\Tr_{n+1/n}(c_{n+1})=a_pc_n-c_{n-1}$, so the trace of $yc_{n+1}+y'c_n\in\hat{E}(\gm_{n+1})$ can be computed with the matrix $A$:
\begin{equation}\label{trace}\Tr_{n+1/n}(c_{n+1},c_n)\hidari y\\ y' \migi = (c_n,c_{n-1})A\hidari y\\ y' \migi \in \hat{E}(\gm_n).\end{equation}
 
\[\text{Thus }(\delta_n^{i+1},\delta_n^i)=(c_n,c_{n-1})AY_{i-1}B_i=\Tr_{n+1/n}(c_{n+1},c_n)Y_{i-1}B_i= \Tr_{n+1/n}(\delta_{n+1}^i, \delta_{n+1}^{i-1})B_i .\] 
\end{proof}
%\text{ where $B_i$ divides one side by $p$:}

%We thus define $Y_i:=A^iB_1...B_i$ for $i\geq0$ and $Y_i:=A^iB_{0}^{-1}B_{-1}^{-1}...B_{i+1}^{-1}$ for $i<0, $ (or equivalently by $AY_{i-1}B_i=Y_i$ and two-way induction), and put
%\[(\delta_n^{i+1},\delta_n^i):=(c_n,c_{n-1})Y_i \text{ for \emph{any} $i\in \Z$}.\]

\begin{lemma}\label{triangle}
\[Y_i=\quad
\begin{cases}\quad
 Y_{i-1}B_+A \text{ if $i$ is even,}\\\quad
Y_{i-1}AB_-\text{ if $i$ is odd.}
\end{cases}\]
\end{lemma}
\begin{proof} We do the proof for odd $i$ by two-way induction (if necessary, just read backwards): 

$Y_i=Y_{i-1}AB_-\iff  Y_{i+1}=AY_iB_+=AY_{i-1}AB_-B_+\iff Y_{i+1}=AY_{i-1}B_-B_+A=Y_iB_+A$
\end{proof}

\begin{lemma} \label{bettertriangle}$Y_i=p^{-\frac{i}{2}} \times A^i$ for i even, and $Y_iB_+A=p^{-\frac{i+1}{2}}\times A^{i+1}$ for $i$ odd.
\end{lemma}
\begin{proof}Two-way induction as in the above proof. Note that $(AB_-B_+A)=p^{-1}\times A^2$.
\end{proof}

\begin{lemma}[Integrality Lemma]\label{matrixcalculation}
$Y_i \in \GL_2(\Z)$, i.e. $Y_i$ has integral coefficients and is invertible.
\end{lemma}

\begin{proof}
Since $\frac{a_p}{p}\in \Z,$ this is true for $i=\pm1$. Now note that $A^2$ is an element of $GL_2(\Z)$ multiplied by $p$. Now use Lemma \ref{bettertriangle} and Lemma \ref{triangle}, noting that $AB_-\in GL_2(\Z)$ and $\det(Y_i)=\pm1$.
%=\dagger B_+^{-1}B_-^{-1}, A^{-2}=\dagger^{-1}B_-B_+$ where $p=B_+^{-1}B_-^{-1}$ and $\frac{1}{p}=B_-B_+$ commute with any other matrix, and $\dagger$ is in $GL_2(\Z)$.
\end{proof}

\begin{corollary}\label{twogenerators}
$\dI$ and $\dn$ generate $\hat{E}(\gm_n)$ as a $\Z_p[\G_n]-$module for $n\geq 0$.
\end{corollary}
\begin{proof}This follows from $c_n$ and $c_{n-1}$ being generators for $\hat{E}(\gm_n)$ and $\det(Y_i)=\pm1$.
\end{proof}

%\begin{corollary} USED?
%$(\delta_n^{i+1}, \delta_n^i)=(\delta_n^i, \delta_n^{i+1})B_+A$ for $i$ even, and

%$(\delta_n^{i+1}, \delta_n^i)=(\delta_n^i, \delta_n^{i+1})AB_-$ for $i$ odd.
%\end{corollary}
%\begin{proof}$(\delta_n^i, \delta_n^{i+1})Y_{i-1}.$
%\end{proof}
\end{section}

\begin{section}{Construction of the zero-finding matrices $\LL_n(X)$}\label{constructionofthetaupsilon}
\begin{definition}\label{kuriharanoPn}

Let $P_{n,x}: H_1(k_n, T) \rightarrow \Z_p[\G_n]$ be defined by

     \[z \mapsto \sum_{\sigma \in \G_n} (x^{\sigma},z)_n\sigma \hspace{5mm}\text{for }x\in \F(\gm_n),\]
where $(\:,\:)_n: \F(\gm_n)\times H^1(k_n,T)\rightarrow H^2(k_n,\Z_p(1))\cong\Z_p$ is the pairing coming from the cup product. Here, we have put $\F(\gm_n)\subseteq H^1(k_n,T)$ (see \cite[Section 8.5]{kobayashisthesis}).

\end{definition}

\begin{definition}\label{defofP}
We let $P_n^i:=P_{n,\delta_n^i}$. \end{definition}
By linearity, the periodicity of $\delta_n^i$ carries over to $P_n^i$, and we also have 
\begin{equation}\label{PY}(P_n^{i+1},P_n^i)=(P_n^1,P_n^0)Y_i.\end{equation}

%\begin{lemma}\label{PPs}USED?
%\[(P_n^{i+1},P_n^i)=\begin{cases}(P_n^i,P_n^{i-1})B_+A\text{ if $i$ is even,}\\(P_n^i,P_n^{i-1})AB_-\text{ if $i$ is odd.}\end{cases}\]
%\end{lemma}
%\begin{proof} Linearity of $P_{n,x}.$
%\end{proof}

The main result of this section is the following proposition.
\begin{propn}\label{mainresultofthissection}Let $z\in H^1(k_n,T)$. Then for some $f_{\sharp}'(z), f_{\flat}'(z)\in \Lambda_n$,
\[(P_n^1(z),P_n^0(z))=(f_{\sharp}'(z), f_{\flat}'(z))\CC_1\cdots\CC_n, \text{ where } \CC_i=\CC_i(X):=\links a_p & \Phi_i(1+X) \\ -1 & 0 \rechts.\]
\end{propn}

The proof of this proposition will occupy the rest of this section. %One key observation is
\begin{observation}\label{observation}Let $k<m$ be integers. Since $\Phi_m(\zeta_{p^k})=1+(\zeta_{p^k})^{p^{m-1}}+...+(\zeta_{p^k})^{p^{m-1}(p-1)}=p$, we have\[\CC_m(\zeta_{p^k}-1)=\links a_p & p \\ -1 & 0\rechts=A.\]
\end{observation}

%\begin{definition} We put $\omega_{\equiv i,\leq m}:=\displaystyle\prod_{1\leq j \leq m, j \equiv i \pmod{\tilde{2}}}\Phi_j(1+X). $
%\end{definition}

\begin{lemma}[Tandem Lemma]\label{importantlemma}Fix an integer $n\geq 0$. Assume that for any $i \in \N$, we are given functions $Q_i=Q_i(X)$ so that $Q_i\in \Phi_i(1+X)\Lambda_n$ whenever $i\leq n$, and $(Q_{n+1}, Q_n)Y_{n'-n}=(Q_{n'+1},Q_{n'})$ for any $n'\in\N$.

Then $(Q_{n+1},Q_n)=(\tilde{q}_1,q_0)\CC_1\CC_2\cdots\CC_n$ with $\tilde{q}_1=\tilde{q}_1(X)\in \Lambda_n, q_0=q_0(X) \in \Lambda_n$.
\end{lemma}

\begin{proof}
We prove that $(Q_{n+1},Q_n)=(\tilde{q}_{n-l+1},q_{n-l})\CC_{n-l+1}\CC_{n-l+2}\cdots\CC_{n}$ for $\tilde{q}_{n-l+1},q_{n-l} \in \Lambda_n$ by induction on $l$. For $l=0$, we put $Q_{n+1}=\tilde{q}_{n+1}, Q_n=q_n$. Now assume this holds for $l\geq0$. We first show $\Phi_{n-l}(1+X)|q_{n-l}(X)$: By Observation \ref{observation} and the inductive hypothesis, 
\[(Q_{n+1},Q_n)=(\tilde{q}_{n+1-l},q_{n-l})A^l \text{ at }X=\zeta_{p^{n-l}}-1,\text
{ and }\] 
\[(Q_{n+1},Q_n)Y_{-l}=(Q_{n+1-l},Q_{n-l}) \text{ by assumption.}\]
But $Q_{n-l}(\zeta_{p^{n-l}}-1)=0,$ so $(\tilde{q}_{n+1-l},q_{n-l})A^lY_{-l}=(Q_{n+1-l},0)$ at $X=\zeta_{p^{n-l}}-1$. 

From Corollary \ref{bettertriangle},
$A^lY_{-l}=p^{\frac{l-1}{2}}\times B_+^{-1}\text{ for $l$ odd,}\text{ and }A^lY_{-l}=p^{\frac{l}{2}}\times I\text{ for $l$ even}.$

Therefore, $q_{n-l}\times (\text{some power of }p)=0$ at $\zeta_{p^{n-l}}-1$. But $\Z_p[\Delta][\zeta_{p^{n-l}}]$ has no $p$-torsion, so we have ${q_{n-l}(\zeta_{p^{n-l}}-1)=0}$. Thus, we can indeed write $q_{n-l}(X)=\Phi_{n-l}(1+X)\tilde{q}_{n-l}(X)$. Since $AB_-\in GL_2(\Z)$, we now define
\[(\tilde{q}_{n-l},q_{n-1-l}):=(\tilde{q}_{n-l+1},\tilde{q}_{n-l})(AB_-)^{-1}.\]
Then $(\tilde{q}_{n-l}, q_{n-l-1})\CC_{n-l}=(\tilde{q}_{n-l},q_{n-l-1})AB_-\links 1 & 0 \\ 0 & \Phi_{n-l}\rechts=(\tilde{q}_{n-l+1},\tilde{q}_{n-l})\links 1 & 0 \\ 0 & \Phi_{n-l}\rechts = (\tilde{q}_{n-l+1},q_{n-l})$. Thus, $(Q_{n+1},Q_n)=(\tilde{q}_{n-l},q_{n-l-1})\CC_{n-l}\cdots\CC_n$, as desired.
\end{proof}

\begin{lemma}\label{zerolemma}
Let $i\leq n$ with $i> 0$. Then $\im(P_n^{i-n}) \subset\Phi_i(1+X) \Lambda_n$. \end{lemma}

\begin{proof}
Consider the ring morphism
\[\xymatrix{\Lambda_i=\Z_p[\Delta][X]/(\omega_i(X))=\Z_p[\G_i]\ar[r]^{\hspace{25mm} \prod\psi}& \prod_\psi\overline{\Q_p}},\]
where $\G_i\longrightarrow^{\hspace{-4mm} \psi} \hspace{3mm}\overline{\Q_p}^\times$ are all the characters of $\G_i$ of conductor $p^{i+1}$ (or $2^{i+2}$ if $p=2$), so that
$\psi(\gamma_i)$ is a primitive $p^i$th root of unity, and thus  $\ker\prod\psi= \Phi_i(1+X)\Lambda_i$.

Denote by $\overline{\sigma}$ the image of $\sigma$ by the natural projection $\G_i\rightarrow\G_{i-1}$. Since $\delta_i^0=c_{i-1}=\delta_{i-1}^1$,%\in\F(\gm_{n-1})
%$\ker\prod\psi$ is generated by $(\omega_{\equiv n, \leq n})\lhd \Lambda_n$. We now prove $\psi \circ P_n^0=0$, from which $\im(P_n^0) \subset (\omega_{\equiv n, \leq n})\lhd\Lambda_n$. From Lemma \ref{deltacongruences} and \cite[Lemma 8.15]{kobayashisthesis}, we have $\pm p^{-[\frac{n+1}{2}]}\psi \circ P_n^0=p^{-[\frac{j+1}{2}]}\psi \circ P_j^0\circ \cor_{n/j}$, whence we can reduce to the case $n=j$: 
\[ \psi \circ P_i^0(z)=\sum_{\sigma \in \G_i}((\delta_{i}^0)^{\sigma},z)_i \psi(\sigma) =\sum_{\tau \in \G_{i-1}}((\delta_{i-1}^1)^{\tau},z)_i \sum_{\sigma \in \G_i, \overline{\sigma}\equiv \tau} \psi(\sigma)=0,\]
so $\im(P_i^0)\subset \ker\psi=\Phi_i(1+X)\Lambda_i$.

From \cite[Lemma 8.15]{kobayashisthesis} and the definition of the $\dn$, 
\[\xymatrix {H^1(k_n,T)\ar[r]^{P_n^{i}}\ar[d]^{\cor}&\Lambda_n\ar[d]^{\proj}\\
H^1(k_{n-1},T)\ar[r]^{\hspace{3mm}p^*P_{n-1}^{i+1}}&\Lambda_{n-1}}
\]
commutes, where $p^*=1$ for odd $i$ and $p^*=p$ for even $i$. The following thus commutes as well:
\[\xymatrix {H^1(k_n,T)\ar[r]^{\hspace{5mm}  P_n^{i-n}}\ar[d]^{\cor}&\Lambda_n\ar[d]^{\proj}\\
H^1(k_{i},T)\ar[r]^{p^*P_i^0}&\Phi_i(1+X)\Lambda_i,}
\]
for an appropriate $p-$power $p^*$. This implies that $P_n^{i-n}$ maps into $\Phi_i(1+X)\Lambda_n$.
\end{proof}

\begin{proof}[Proof of Proposition \ref{mainresultofthissection}]For any $i\in \N$, put $Q_i:=P_n^{i-n}(z)$ and apply the Tandem Lemma \ref{importantlemma}. We know that $Q_i \in \Phi_i(1+X)\Lambda_n$ from Lemma \ref{zerolemma} for $i\leq n$, so by equation (\ref{PY}), 
\[ (Q_{n+1},Q_n)Y_{n'-n}=(P_n^1(z),P_n^0(z))Y_{n'-n}=(P_n^{n'-n+1}(z),P_n^{n'-n}(z))=(Q_{n'+1},Q_{n'}).\]
\end{proof}

\begin{notation}We define matrices $\LL_n$ with entries in $\Lambda_n$ by \[\LL_n=\LL_n(X):=\Aux\CC_1...\CC_n,%=\links a_p & \Phi_1(1+X) \\ -1 & 0 \rechts ... \links a_p & \Phi_n(1+X) \\ -1 & 0 \rechts.\]
\text{ where }\Aux:=-(AB_-)^{-1}=\links 0 & 1 \\ -1 & -a_p \rechts.\]
\end{notation}

\begin{definition}We now put $(f_\sharp(z), f_\flat(z)):=(f_\sharp'(z),f_\flat'(z))\Aux^{-1}$. This might seem unnatural at first, since in Sections \ref{growthproperties}, \ref{colemanmaps}, and  \ref{padiclfunctions}, there are no problems if you formally replace $(f_\sharp(z),f_\flat(z))$ by $(f_\sharp'(z),f_\flat'(z))$ while ignoring $\Aux$. We need this definition to match a sign convention of Kobayashi (Remark \ref{comparison}), and because we need this setup in Section \ref{mainconjecture}. The main proposition then becomes:
\end{definition}

\begin{propn}For $z\in H^1(k_n,T)$, there are $(f_\sharp(z),f_\flat(z))\in\Lambda_n^{\oplus2}$ so that we have \[(P_n^1(z),P_n^0(z))=(f_\sharp(z),f_\flat(z))\LL_n.\]
\end{propn}
\end{section}

\begin{section}{Construction and Growth Properties of $\LL$}\label{growthproperties}

We now scrutinize the growth properties of $\LL_n$:

\begin{definition}
Fix $0<r<1$. For $f(X)\in \C_p[[X]]$ convergent on the open unit disc of $\C_p$, let  
\[|f(X)|_r:=\sup_{|z|_p<r}|f(z)|_p \]
with normalization $|p|_p=\frac{1}{p}$, and for a matrix with such entries, define its norm $\left| \quad \right|_r$ by
\[\left|\links a&b\\c&d\rechts\right|_r:=\max\{|a|_r,|b|_r,|c|_r,|d|_r\}.\]
\end{definition}

\begin{example}\label{basicexample}\footnote{This appears in the proof of \cite[Lemma 4.5]{pollacksthesis}. There seems to be a typo. He meant to write ${|\Phi_n(1+X)|_r\sim r^{p^{n-1}(p-1)}}$.}\[\left|\Phi_n(1+X)\right|_r\quad=\quad\begin{cases}\frac{1}{p} &\text{when }   r\leq p^{-\frac{1}{p^{n-1}(p-1)}},\\ r^{p^{n-1}(p-1)}&\text{when }  r \geq p^{-\frac{1}{p^{n-1}(p-1)}}.\end{cases}\]
\end{example}

\begin{example}\label{matrixexample}$|\frac{1}{p}\CC_n(X) \CC_{n+1}(X)|_r=|\frac{\Phi_n(1+X)}{p}|_r=1$ for $n$ big enough so that $p^n\geq \frac{p}{p-1}\frac{\log p}{\log {r^{-1}}}$.
\end{example}
\begin{proof}For appropriately big $n$,
\[\left|\frac{1}{p}\CC_n\CC_{n+1}\right|_r=\left|\links \frac{a_p^2-\Phi_n(1+X)}{p}& \frac{a_p\Phi_{n+1}(1+X)}{p}\\ -\frac{a_p}{p} & -\frac{\Phi_{n+1}(1+X)}{p}\rechts \right|_r=\left|\frac{\Phi_n(1+X)}{p}\right|_r=1.\] \end{proof}

\begin{lemma}[Convergence Lemma]
$\displaystyle\lim_{n\to\infty}\LL_nA^{-n}$ exists and converges on the open unit disc of $\C_p$.
\end{lemma}
\begin{proof}
Using the norm $\left|\:\:\right|_r$, \[R:=\left|\LL_nA^{-n}-\LL_{n+1}A^{-(n+1)}\right|_r=\left|\LL_nA^{-n}-\LL_n\CC_{n+1}A^{-(n+1)}\right|_r\] satisfies \[R\leq\left|p^{-\frac{n}{2}}\LL_n\right|_r\left|p^{\frac{n}{2}}(A^{-n}-\CC_{n+1}A^{-(n+1)})\right|_r\leq\left|p^{-\frac{n}{2}}\LL_n\right|_r\left|I-\CC_{n+1}A^{-1}\right|_r\left|p^{\frac{n}{2}}A^{-n}\right|_r.\]

As $n\rightarrow \infty$, the first term $\left|p^{-\frac{n}{2}}\LL_n\right|_r$ is bounded by Example \ref{matrixexample}, and analogous calculations as in its proof yield $\left|pA^{-2}\right|_r=1$, so the last term $\left|p^{\frac{n}{2}}A^{-n}\right|_r$ is bounded, too. Finally, from \cite[proof of Lemma 4.1]{pollacksthesis}, we have $\left|\frac{\Phi_{n+1}(1+X)}{p}-1\right|_r\rightarrow 0$ as $n\rightarrow \infty$, so the middle term $\left|I-\CC_{n+1}A^{-1}\right|_r\rightarrow 0$ as $n \rightarrow \infty$. It follows that $\left|\LL_nA^{-n}-\LL_{n+1}A^{-(n+1)}\right|_r\rightarrow 0$ as $n \rightarrow \infty$, whence the lemma.\end{proof}

\begin{definition}\label{defofLL}
Recall that $N=n+1$ for $p$ odd and $N=n+2$ for $p$ even. Now let us put 
\[ \LL:=\lim_{n\rightarrow \infty}\LL_nA^{-N}.\]
\end{definition}

\begin{definition}Let $f(X), g(X)\in \C_p[[X]]$ converge on the open unit disc of $\C_p$. Then we say that $f(X)$ is $O(g(X))$ if \[\left|f(X)\right|_r\text{  is }O(\left|g(X)\right|_r)\text{  as }r\rightarrow 1^-.\]
If in addition, $g(X)$ is $O(f(X))$, then we say that $f(X)\sim g(X)$.
\end{definition}

\begin{example}\label{growthexample}$1\sim X$. Since $\log_p(1+X)=X\displaystyle\prod_n\frac{\det\CC_n}{\det A}$, we have $\det \LL\sim \log_p(1+X)$.
\end{example}
\begin{lemma}[Growth Lemma]\label{growthlemma}The entries of $\LL$ are $O(\log_p(1+X)^{\frac{1}{2}})$.
\end{lemma}
\begin{proof} We give the proof for $N=n+1$, since it is very similar for $N=n+2$.
Note that
\[\LL_nA^{-N}=\Aux\CC_1...\CC_nA^{-N}=\begin{cases}\Aux\CC_1...\CC_nY_{-N}\times p^{-\frac{n+1}{2}}&\text{ for $n$ odd,}\\\Aux\CC_1...\CC_nY_{-N}B_+\times p^{-\frac{n}{2}}&\text{ for $n$ even (cf. Corollary \ref{bettertriangle}).}\end{cases}\]
\[\text{Thus, }\left|\LL_nA^{-N}\right|_r\leq\begin{cases}\left|\Aux\CC_1\right|_r\left|\frac{1}{p}\CC_2\CC_3\right|_r...\left|\frac{1}{p}\CC_{n-1}\CC_n\right|_r\left|Y_{-N}\times p^{-1}\right|_r &\text{ for $n$ odd,}\\\left|\Aux\right|_r\left|\frac{1}{p}\CC_1\CC_2\right|_r...\left|\frac{1}{p}\CC_{n-1}\CC_n\right|_r\left|Y_{-N}B_+\right|_r &\text{ for $n$ even.}\end{cases}\]

From Example \ref{matrixexample}, $\left|\frac{1}{p}\CC_k\CC_{k+1}\right|_r=\left|\frac{\Phi_k(1+X)}{p}\right|_r$, so \[\left|\LL_nA^{-N}\right|_r\leq{\left|\displaystyle\prod_{1\leq k < n, k\not\equiv n(2)}\frac{\Phi_k(1+X)}{p}\right|_r\times c}\] for some constant $c$ independent from $r$. From \cite[Lemma 4.5]{pollacksthesis}, we have  
\[\log_p(1+X)^{\frac{1}{2}} \sim \frac{1}{p}\prod_{k \text{ odd}} \frac{\Phi_k(1+X)}{p} \sim \frac{1}{p} \prod_{k \text{ even}} \frac{\Phi_k(1+X)}{p},\]
so the entries of $\LL$ are all $O(\log_p(1+X)^{\frac{1}{2}})$, as desired.
\end{proof}
\end{section}

\begin{section}{Construction of the Coleman maps}\label{colemanmaps}
In this section, we construct a map $\col: \mathbf{H}^{1}_{\Iw}(T) \rightarrow \Lambda \oplus \Lambda$ , where $\mathbf{H}^{1}_{\Iw}(T)=\displaystyle\varprojlim_n H^1(k_n,T)$.
\begin{definition}Let $\llll_n^i$ be the $\lord$-module morphisms given by
\[\begin{array}{cccccccc}
 \Lambda_n \oplus  \Lambda_n& \stackrel{\llll_n^i}{\longrightarrow}& (\Lambda_n \oplus \Lambda_n)\LL_nY_i\subset &\Lambda_n \oplus \Lambda_n.
 \\
(a  , b)& \longmapsto & (a,b)\LL_nY_i
\end{array}\]
\end{definition}

\begin{observation}\label{independenceofi}$\ker \llll_n^i$ is independent from $i$, and we henceforth denote it by $\ker \llll_n$.
\end{observation}
\begin{proof}Right multiplication by $Y_i$ is a $\lord$-module isomorphism, because $\det Y_i=1$.
\end{proof}

\begin{propn}\label{col}There is a unique homomorphism $\col_n$ so that the following commutes:
\[\xymatrix {H^1(k_n,T)\ar@/^8mm/@{-->}[0,2]^{\hspace{2.5mm}\exists !\col_n}\ar[r]^{(P_n^{i+1}, P_n^i)}& {\Lambda_n\oplus \Lambda_n} &\ar@{_(->}[l]_{\llll_n^i}\dfrac{\Lambda_n \oplus \Lambda_n}{\ker \llll_n}
}
\] \end{propn}

\begin{proof}For $z\in H^1(k_n,T),(P_n^{i+1}(z),P_n^i(z))=(f_\sharp(z),f_\flat(z))\LL_nY_i= \llll_n^i(f_\sharp(z),f_\flat(z))$. Thus, we can put $\col_n(z)=(f_\sharp(z),f_\flat(z))\in \frac{\Lambda_n\oplus\Lambda_n}{\ker \llll_n}$, independent from $i$.
\end{proof}

%From the above proof, $\im (h_{n+1}^{i-1}) \subset \im (P_{n+1}^{i},P_{n+1}^{i-1})$. 
%Now $\proj$ is well-defined, and $\bb_i$ will be well-defined on $\im (h_{n+1}^{i-1}) \supset \im (P_{n+1}^{i},P_{n+1}^{i-1})$:

\begin{lemma} $\proj: \Lambda_{n+1}\oplus\Lambda_{n+1}\rightarrow \Lambda_n\oplus \Lambda_n$ satisfies $\proj(\ker \llll_{n+1})\subset \ker \llll_n$.
\end{lemma}
\begin{proof}For $(a,b)\in \ker \llll_{n+1}\subset \Lambda_{n+1}^{\oplus2}$, we know that $(a,b)\LL_{n+1}=(0,0)$ in $\in \Lambda_{n+1}^{\oplus2}$.

Now $\Phi_{n+1}(1+X)\equiv p \pmod{\omega_n(X)} \text{ implies }
(a,b)\LL_{n+1}\equiv(a,b)\LL_nA\equiv (0,0) \pmod{\omega_n}, \text{ so}$
 \[p(a,b)\LL_n\equiv(0,0)\pmod{\omega_n},\]since $pA^{-1}$ has integral coefficients. But $\Lambda_{n+1}/\omega_n\cong \Lambda_n$ has no $p$-torsion, so $(a,b)\LL_n \in \Lambda_n^{\oplus 2}$. Thus, $\proj\left((a,b)\right) \in \ker \llll_n$.%\renewcommand{\qedsymbol}{{$\mathit{QED\text{ } for\text{ } Lemma}$}}
\end{proof}
\renewcommand{\qedsymbol}{{$\mathit{QED}$}}
Now we can define $\bb_i: \im(h_{n+1}^{i-1}) \rightarrow \im(h_n^i)$ (cf. diagram (8.27) in \cite{kobayashisthesis}) as
\[\bb_i=\begin{cases}(\proj,\frac{1}{p}\proj) & \text{ for $i$ odd,}\\(\frac{1}{p}\proj, \proj) & \text{ for $i$ even.}\end{cases}\]

\begin{propn}\label{prop2}The following diagram commutes: 
\[\xymatrix 
{H^1(k_{n+1},T)\ar[0,2]^{\hspace{-5mm}P_{n+1}^i,P_{n+1}^{i-1}}\ar[d]^{\cor}
 & & \hspace{0mm}\im (h_{n+1}^{i-1})\ar[d]^{\bb_i}
 &\dfrac{\Lambda_{n+1} \oplus \Lambda_{n+1}}{\ker \llll_{n+1}}\ar[d]^{\proj}\ar@{_(->}[l]^{\hspace{0mm}\llll_{n+1}^{i-1}}\\
H^1(k_n,T)\ar[0,2]^{\hspace{-7mm}P_n^{i+1},P_n^i}
& &\im (h_n^i)
& \dfrac{\Lambda_n \oplus \Lambda_n}{\ker \llll_n}\ar@{_(->}[l]^{\hspace{0mm}\llll_n^i}}
\]

\end{propn}
\begin{proof}As for the left square, commutativity follows from the trace compatibility of the $P_n^i$ (see \cite[Lemma 8.15]{kobayashisthesis}). The map $\bb_i$ ensures that we divide by the appropriate power of $p$ in accordance with Definition \ref{bestdefinition} of $\delta_n^i$. 

As for the right square, the map $\bb_i\circ \llll_{n+1}^{i-1}$ is given by multiplying by the matrix $\LL_{n+1}Y_{i-1}B_i \equiv \LL_nAY_{i-1}B_i \pmod{\omega_n(X)}$. But $AY_{i-1}B_i=Y_i$ by construction, so we are done.
\end{proof}

Combining Propositions \ref{col} and \ref{prop2}, we obtain the following corollary.
\begin{corollary}\label{cor} The Coleman maps are compatible:
\[\xymatrix {H^1(k_{n+1},T)\ar[r]^{\col_{n+1}}\ar[d]^{\cor}\ar@{}[dr]|\circlearrowleft& \dfrac{\Lambda_{n+1} \oplus \Lambda_{n+1}}{\ker \llll_{n+1}}\ar[d]^{\proj}\\
H^1(k_n,T)\ar[r]^{\col_n}& \dfrac{\Lambda_n \oplus \Lambda_n}{\ker \llll_n}}
\]
\end{corollary}

\begin{propn}[Limit Proposition]\label{limitpropn}
\[ \varprojlim_{n} \frac{\Lambda_n \oplus \Lambda_n}{\ker \llll_n} \cong \Lambda \oplus \Lambda. \]
\end{propn}
\begin{proof}The map $h_n^i$ can be defined on $\Lambda^{\oplus2}$ as well, as multiplication by the matrix $\H_nY_i$. Let $\pi_n$ be the projection $\Lambda\rightarrow \Lambda_n$ and put $M_n:=\ker\pi_n^{\oplus 2}\circ h_n^i$ (which is independent of $i$, cf. Observation \ref{independenceofi}). By definition, $\omega_n(X)\Lambda^{\oplus 2}\subset M_n$, from which $\frac{\Lambda\oplus \Lambda}{M_n}\cong\frac{\Lambda_n\oplus \Lambda_n}{\ker h_n}$ as $\Lambda$-modules. Now
\[\Lambda^{\oplus 2}/M_n \cong \varprojlim_m \Lambda^{\oplus 2}/(M_n+p^m\Lambda^{\oplus2}),\]
since $M_n$ is a free finitely generated $\Z_p$-module. 
We have $M_n \supset\omega_n\Lambda^{\oplus 2}$, so $M_n+p^m\Lambda^{\oplus 2}\supset \omega_n\Lambda^{\oplus 2}+p^m\Lambda^{\oplus 2}$. Assume for the moment that we also have $M_{2m+\nu}+p^m\Lambda^{\oplus 2}\subset\omega_\nu\Lambda^{\oplus 2}+p^m\Lambda^{\oplus 2}.$ These two inclusions would induce module morphisms
\[\varprojlim_{(\nu,m)}\Lambda^{\oplus 2}/(\omega_\nu\Lambda^{\oplus 2}+p^m\Lambda^{\oplus 2})\rightarrow \varprojlim_{(n,m)}\Lambda^{\oplus 2}/(M_n+p^m\Lambda^{\oplus 2}),\]
\[ \varprojlim_{(n,m)}\Lambda^{\oplus 2}/(M_n+p^m\Lambda^{\oplus 2})\rightarrow\varprojlim_{(\nu,m)}\Lambda^{\oplus 2}/(\omega_\nu\Lambda^{\oplus 2}+p^m\Lambda^{\oplus 2}),\]
which are inverses of each other, from which $\displaystyle\varprojlim_n\Lambda^{\oplus 2}/M_n=\Lambda\oplus \Lambda$ would follow. Let us thus prove this second inclusion. It suffices to prove that $M_{2m+\nu}=0$ inside $\Lambda_\nu^{\oplus 2}/p^m$. 
\begin{lemma}$\LL_n^{-1}\times \omega_n(X)$ has coefficients in $\lord$, and $M_n=(\Lambda_n\oplus\Lambda_n)\LL_n^{-1}\times \omega_n(X)$.
\end{lemma}
\begin{proof}Choose $(a,b)\in M_n\subset\Lambda^{\oplus 2}$. This says that $(a,b)\LL_n=(c,d)\omega_n$ for some $(c,d)\in\Lambda^{\oplus 2}$. Conversely, for a fixed $(c,d)\in\Lambda^{\oplus 2}$, any such $(a,b)$ is unique, since $\det \LL_n=\frac{\omega(X)}{X}\neq 0$ in $\Z_p[X]$ , and this is not a zero-divisor in $\lord$. Thus, $(a,b)=(c,d)\LL_n^{-1}\omega_n$.
\renewcommand{\qedsymbol}{{$\mathit{QED\text{ } for\text{ } Lemma}$}}
\end{proof}
\renewcommand{\qedsymbol}{{$\mathit{QED\text{ } for\text{ } Limit\text{ } Proposition}$}}
%If $(a,b)\in M_{2m+\nu}$, then $(a,b)\Aux\LL_{2m+\nu}=(c,d)\times\omega_{2m+\nu}$. From Observation \ref{observation}, we have $(a,b)\Aux\LL_{\nu}A^{2m}\equiv(c,d)\times p^{2m}\omega_{\nu} \pmod{\omega_{\nu}}$. Thus as elements of $\Lambda_{\nu}, (a,b)=(c,d)p^{2m}\omega_{\nu}A^{-2m}\LL_{\nu}^{-1}\Aux^{-1}=(p^mc,p^md)Y_{-2m}\spadesuit\Aux^{-1}$ with $\spadesuit \in X\times \MM_2(\Lambda_{\nu})$by Lemma \ref{bettertriangle}. We note that $\omega_{\nu}=(\det\LL_{\nu})/X$ is not a zero-divisor of $\Z_p[\Delta][[X]]$ since $\omega_{\nu} \neq0$ in $\Z_p[X]$. Hence, $M_{2m+\nu}\subset p^m\Lambda_{\nu}^{\oplus 2}.$
Now $\LL_{2m+\nu}^{-1}\times\omega_{2m+\nu}\equiv(\LL_{\nu}\CC_{1+\nu}...\CC_{2m+\nu})^{-1}\omega_{\nu}\Phi_{1+\nu}...\Phi_{2m+\nu}\equiv A^{-2m}\LL_{\nu}^{-1}\omega_{\nu}\times p^{2m} \pmod{\omega_{\nu}}$ by Observation \ref{observation}. But $pA^{-2}\in\MM_2(\Z),$ so $\LL_{2m+\nu}^{-1}\times \omega_{2m+\nu}\in p^m\Lambda_{\nu}^{\oplus 2}$ and thus $M_{2m+\nu}=0$ in $\Lambda_{\nu}^{\oplus 2}/p^m$.
\end{proof}
\renewcommand{\qedsymbol}{{$\mathit{QED}$}}

In view of the Limit Proposition \ref{limitpropn} and Corollary \ref{cor}, we can make the following definition:
\begin{definition} Let the Coleman map $\col: \mathbf{H}^{1}_{\Iw}(T) \rightarrow \Lambda \oplus \Lambda$ be the projective limit of $\col_n: H^1(k_n,T)\rightarrow\dfrac{\Lambda_n \oplus \Lambda_n}{\ker h_n},\text{ where } \Lambda = \Z_p[\Delta][[X]].$ 
\end{definition}
\end{section}

\begin{section}{The $p$-adic $L$-functions $L_p^{\sharp}(E,X)$ and $L_p^{\flat}(E,X)$}\label{padiclfunctions}
Recall (e.g. from \cite[6.3]{washington}) that for a character $\eta:\Delta \rightarrow \Z_p^{\times}$, a $\Z_p[\Delta]$-module $M$, and an odd prime $p$, the $\eta$-component $M^{\eta}$ is simply given by $\varepsilon_{\eta}M$, where $\displaystyle{\varepsilon_{\eta}=\frac{1}{\#\Delta} \sum_{\tau\in \Delta} \eta(\tau)\tau^{-1}}$.

\begin{definition}For $p$ odd, let $\mathbf{z^\pm}=(z_n^\pm)_n\in\mathbf{H}^{1}_{\Iw}(T)=\displaystyle\varprojlim H^1(k_n,T)$ be Kato's zeta elements (see \cite[Theorem 5.2 i]{kobayashisthesis} and \cite[Theorem 12.5]{kato}). Write \[(L_p^{\sharp}(E,\eta,X),L_p^{\flat}(E,\eta,X))\] for the $\eta$-component of the image of $\eta(-1)\col(\mathbf{z}^{\eta(-1)})$ in $\Lambda^2,$ which we naturally view as an element of $\Z_p[[X]]^2\cong(\Lambda^\eta)^2$. If $p=2$, we have $\mathbf{z^\pm}=(z_n^\pm)_n\in\mathbf{H}^{1}_{\Iw}(T)\otimes \Q=\displaystyle\varprojlim H^1(k_n,T)\otimes \Q$, and we extend our Coleman map naturally to $\col_\Q: \mathbf{H}^{1}_{\Iw}(T)\otimes \Q\rightarrow (\Lambda\oplus \Lambda)\otimes \Q$. Now define \[(L_2^{\sharp}(E,\eta,X),L_2^{\flat}(E,\eta,X))\]as the image of $(\eta(-1)\col_\Q(\mathbf{z}^{\eta(-1)}))$ in the quotient $\Q\otimes\Z_2[[X]]^2$ of $\Q\otimes \Z_2[\Delta][[X]]^2=\Q\otimes\Lambda^{\oplus2}$  induced by the map $\Lambda\rightarrow \Z_2[[X]], g\in \Delta \mapsto \eta(g)$.
\end{definition}

Let $\psi=\eta\chi$ be a non-trivial character of $\G_n$ of conductor $N$, so that $\chi(\gamma_n)$ is a primitive $p^n$th root of unity. We also let $\alpha$ and $\overline{\alpha}$ be the two roots of $X^2-a_pX+p$. Recall that $N=n+1$ for $p$ odd, and $N=n+2$ if $p=2$. Let $L(E,\psi,s)$ be the Hasse-Weil $L$-function with Dirichlet character $\psi$, and denote by $\Omega_E^\pm$ the real and imaginary N\'{e}ron periods, obtained by integrating an invariant differential\begin{footnote}{which is defined up to multiplication by $\pm 1$. We use the same $\omega_E$ for constructing $\Omega^\pm$ and $\exp^*_{\omega_E}$.}\end{footnote}$\omega_E$. We can compare $L(E,\psi,s)$ with its $p$-adic counterpart $L_p(E,\alpha,\eta,X)$ of Mazur, Tate, and Teitelbaum originally due to \label{amicevelu}Amice and V\'{e}lu\cite{amicevelu}, and Vi\v{s}ik\cite{vishik}:

\begin{theorem}\label{amicevelu}Let the notation be as above. For the $p$-adic $L$-function $L_p(E,\alpha,\eta,X)$ of \cite{MTT},
\[\alpha^NL_p(E,\alpha,\eta,\chi(\gamma_n)-1)=\frac{p^{N}}{\tau(\overline{\psi})}\frac{L(E,\overline{\psi},1)}{\Omega_E^{\eta(-1)}}, \]
where $\tau(\overline{\psi})$ is the Gau{\ss} sum $\displaystyle \sum_{\sigma\in\G_n}\overline{\psi}(\sigma)\zeta_{p^N}^{\sigma}$.

For $\psi=1$, we have 
\[L_p(E,\alpha,0)=(1-\alpha^{-1})^2\frac{L(E,1)}{\Omega_E^+}.\]
\end{theorem}
The statement of this theorem only makes sense after fixing an embedding of an algebraic closure $\overline{\Q}$ of $\Q$ inside the completion of an algebraic closure of $\Q_p$: We fix an embedding $\overline{\Q} \rightarrow \C_p$. We also fix an embedding $\overline{\Q} \rightarrow \C$.
We would like to go further and compare both of these very classical $p$-adic $L$-functions to our $p$-adic $L$-functions $L_p^{\sharp}$ and $L_p^{\flat}$.

There is a map $\exp_{\omega_E}^*:H^1(k_n,V)\rightarrow \text{cotan}(E/k_n),$ where $V=T\otimes \Q_p$ and $\text{cotan}(E/k_n)$ is the cotangent space of $E/k_n$ at the origin (see \cite[Section 8.7]{kobayashisthesis}) with the following properties:
\begin{propn}\label{Pandlogandexp} The morphism $P_{n,x}$ in Definition \ref{kuriharanoPn} is given by
\[P_{n,x}(z)=\left(\sum_{\sigma\in\G_n}\log_{\F}(x^\sigma)\sigma\right)\left(\sum_{\sigma\in\G_n}\exp_{\omega_E}^*(z^\sigma)\sigma^{-1}\right).\]
\end{propn}
This is \cite[Proposition 8.25]{kobayashisthesis}. The following theorem is \cite[Theorem 12.5]{kato}:
\begin{theorem}[Kato]\label{kato}  With notation as above,
\[\psi\left(\sum_{\sigma\in\G_n}\exp_{\omega_E}^*(\mathbf{z}^{\eta(-1)\sigma})\sigma^{-1}\right)=\frac{L(E,\overline{\psi},1)}{\Omega_E^{\eta(-1)}}.\]
\end{theorem}

\begin{propn}\label{yilemma}Let $\zeta_{p^n}$ be any primitive $p^n$th root of unity. With notation as above,
\[\left(\alpha^NL_p(E,\alpha,\eta,\zeta_{p^n}-1),0\right)=\left(L_p^{\sharp}(E,\eta,\zeta_{p^n}-1),L_p^{\flat}(E,\eta,\zeta_{p^n}-1)\right)\LL_n(\zeta_{p^n}-1).\]
\end{propn}

\begin{proof}%[Proof of proposition \ref{yilemma}]
Let $\pi_{\psi}:\Lambda_n\rightarrow \overline{\Q_p}$ be induced by a character $\psi:\G_n\rightarrow \overline{\Q}^{\times}\subset \overline{\Q_p}^{\times}$ which sends $X$ to $\zeta_{p^n}-1$. Then by definition, the right hand side equals \[\pi_{\psi}^{\oplus 2}\bigg(\llll_n^0\left(\eta(-1)\col_n(z_n^{\eta(-1)})\right)\bigg)=\bigg(\psi\circ P_n^1(z_n^{\eta(-1)}), \psi \circ P_n^0(z_n^{\eta(-1)})\bigg)\times \psi(-1).\]
But by Proposition \ref{Pandlogandexp},\[\psi \circ P_n^i(z_n^{\eta(-1)})=\psi\left(\displaystyle\sum_{\sigma\in\G_n}\log_{\F}(\delta_n^i)^\sigma\sigma\right)\psi\left(\displaystyle\sum_{\sigma\in\G_n}\exp_{\omega_E}^*({z_n}^{\eta(-1)\sigma})\sigma^{-1}\right), \text{ so }\]

\[\left(\psi \circ P_n^1(z_n^{\eta(-1)}), \psi \circ P_n^0(z_n^{\eta(-1)})\right)=\left(\tau(\psi)\frac{L(E,\overline{\psi},1)}{\Omega_E^{\eta(-1)}},0\right)\] by Theorem \ref{kato} and trace computations, which equals

\[=\left(\frac{\psi(-1)p^N}{\tau(\overline{\psi})}\frac{L(E,\overline{\psi},1)}{\Omega_E^{\eta(-1)}},0\right)=\left(\psi(-1)\alpha^NL_p(E,\alpha,\eta,\zeta_{p^n}-1),0\right)\text{ by Theorem \ref{amicevelu}}.\]\end{proof}

\begin{corollary}\label{doubletriangle}$\left(L_p^{\sharp}(E,\eta,\zeta_{p^n}-1),L_p^{\flat}(E,\eta,\zeta_{p^n}-1)\right)\LL(\zeta_{p^n}-1)=\left(\alpha^NL_p(E,\eta,\alpha,\zeta_{p^n}-1),0\right)A^{-N}$
\end{corollary}
\begin{proof}We have $\LL(\zeta_{p^n}-1)=\displaystyle\lim_{m\rightarrow\infty}\LL_n\CC_{n+1}...\CC_{n+m}A^{-m}A^{-N}(\zeta_{p^n}-1).$
But $\CC_{n+m}(\zeta_{p^n}-1)=A$ for ${m>0}$ from Observation \ref{observation}, so $\LL(\zeta_{p^n}-1)=\LL_n(\zeta_{p^n}-1)A^{-N}.$
\end{proof}

\begin{lemma}\label{marumaru} We have the matrix identity $A^{-N}\links 1 & 1 \\ -\alpha^{-1} & -\overline{\alpha}^{-1}\rechts = \links \alpha^{-N} & \overline{\alpha}^{-N}\\ -\alpha^{-N-1}& -\overline{\alpha}^{-N-1}\rechts$.
\end{lemma}
\begin{proof} We obtain this by diagonalization or by induction. Note that $1=\frac{a_p}{p}\alpha-\frac{1}{p}\alpha^2$.
\end{proof}

\begin{definition} Put $\links \log_\alpha^{\sharp}(1+X) & \log_{\overline{\alpha}}^{\sharp}(1+X) \\ \log_\alpha^{\flat}(1+X) & \log_{\overline{\alpha}}^{\flat}(1+X) \rechts:=\LL(X)\links 1 & 1 \\ -\alpha^{-1}& -\overline{\alpha}^{-1}\rechts.$ \end{definition}

%\begin{definition}For locally analytic functions $F$ and $G$ on the open unit disc of $\C_p$, we write $F\sim G$ if $F\inO(G)$ and $G\inO(F)$.
%\end{definition}

\begin{lemma}We have \[\log_\alpha^{\sharp}(1+X)\sim \log_{\overline{\alpha}}^{\flat}(1+X)\sim\log_p(1+X)^{\frac{1}{2}}\text{ or }\log_\alpha^{\flat}(1+X)\sim \log_{\overline{\alpha}}^{\sharp}(1+X)\sim\log_p(1+X)^{\frac{1}{2}}.\]
\end{lemma}

\begin{proof}
From the Growth Lemma \ref{growthlemma}, $\log_\alpha^*=\log_\alpha^*(1+X)\in O(\log_p(1+X)^{\frac{1}{2}})$ for $*\in\{\sharp,\flat\}$. \[\begin{array}{ll}\det \links \log_\alpha^{\sharp} & \log_{\overline{\alpha}}^{\sharp} \\ \log_\alpha^{\flat} & \log_{\overline{\alpha}}^{\flat} \rechts& =\det (\LL) \det \links 1 & 1 \\ -\alpha^{-1} &-\overline{\alpha}^{-1}\rechts \sim \log_p(1+X), \text{ cf. Example \ref{growthexample},}\end{array}\] so $\log_\alpha^{\sharp}(1+X)\log_{\overline{\alpha}}^{\flat}(1+X)-\log_{\overline{\alpha}}^{\sharp}(1+X)\log_\alpha^{\flat}(1+X)\sim\log_p(1+X).$
\end{proof}
\begin{definition}Let $K$ be a finite extension of $\Q_p$. We let \[\mathcal{A}(K):=\left\{f\in K[[X]]\;  \big| \;f \text{ is convergent on the open unit disc of }\C_p\right\}.\]
\end{definition}
\begin{lemma}[Interpolation Lemma]\label{interpolationlemma}Let $0\leq h <1$ be a real number, and let $f(X),g(X)\in\mathcal{A}(K)$. If $f(X)\in O(\log_p(1+X)^h)$ and $g(X) \in O(\log_p(1+X)^h)$ satisfy $f(\zeta_{p^n}-1)=g(\zeta_{p^n}-1)$ for all primitive $p^n$th roots of unity with $n \geq 1$, then $g(X)=f(X)$. 
\end{lemma}
\renewcommand{\qedsymbol}{{$\textit{QEA}$}}
\begin{proof}
Put $h(X):=f(X)-g(X) \in O(\log_p(1+X)^h)$. Then $h(\zeta_{p^n}-1)=0$, so by \cite[Lemma 4.7]{lazard}, $\frac{\log_p(1+X)}{X}|h(X)$ in $\mathcal{A}(K)$. If $h(X)\neq 0$, this would yield $\frac{\log_p(1+X)}{X}\sim \log_p(1+X)\in O(h(X))$ (cf. \cite[discussion before Proposition 2.11]{pollacksthesis}, \cite[Proposition I.4.5]{colmez}), and thus $\log_p(1+X)\in O(\log_p(1+X)^h)$,
\end{proof}

We are now ready to state the main theorem:
\begin{maintheorem}
\[\log^{\sharp}_\alpha(1+X)L_p^{\sharp}(E,\eta,X)+\log^{\flat}_\alpha(1+X)L_p^{\flat}(E,\eta,X)=L_p(E,\alpha,\eta,X)\]
\end{maintheorem}
\renewcommand{\qedsymbol}{{$\textit{QED}$}}
\begin{proof}Put $L_p^*=L_p^*(E,\eta,\zeta_{p^n}-1), \log_\alpha^*=\log_\alpha^*(\zeta_{p^n})$ for $*\in\{\sharp,\flat\}$, and $L_p(\alpha)=L_p(E,\alpha,\eta,\zeta_{p^n}-1)$. From Corollary \ref{doubletriangle} and Lemma \ref{marumaru}, we have
\[(L_p^{\sharp},L_p^{\flat}) \links \log_\alpha^{\sharp} & \log_{\overline{\alpha}}^{\sharp} \\ \log_\alpha^{\flat} & \log_{\overline{\alpha}}^{\flat} \rechts =(\alpha^NL_p(\alpha),0)\links \alpha^{-N} & \overline{\alpha}^{-N}\\ -\alpha^{-N-1} & \overline{\alpha}^{-N-1}\rechts = (L_p(\alpha), L_p(\overline{\alpha})),\]
since by Theorem \ref{amicevelu}, $\alpha^NL_p(\alpha)=\overline{\alpha}^NL_p(\overline{\alpha})$. Because this holds for all primitive $p^n$th roots of unity $\zeta_{p^n}$ and both sides are $O(\log_p(1+X)^{\frac{1}{2}})$, they agree by the Interpolation Lemma \ref{interpolationlemma} above with $h=\frac{1}{2}$.
\end{proof}

\begin{remark} By convention (e.g. \cite{pollacksthesis}), we suppress the character $\eta$ from the notation if it is trivial. Thus for the trivial character, we have
\[\log^{\sharp}_\alpha(1+X)L_p^{\sharp}(E,X)+\log^{\flat}_\alpha(1+X)L_p^{\flat}(E,X)=L_p(E,\alpha,X)\]
as in the introduction.
\end{remark}

%By a theorem of Rohrlich, $L_p(E,\alpha,X)$ is nonzero (\cite{rohrlich}), and thus we can say
%\begin{propn}At least one of $L_p^{\sharp}(E,\eta,X)$ and $L_p^{\flat}(E,\eta,X)$ is nonzero.
%\end{propn}
%\begin{lemma}Suppose that $a_p=0$. Then both $L_p^{\sharp}(E,\eta,X)$ and $L_p^{\flat}(E,\eta,X)$ are nonzero.
%\end{lemma}
%\begin{proof}This follows e.g. from \cite[page 7]{kobayashisthesis} by looking at the values $L_p^*(E,\eta,\zeta_{p^n}-1).$\end{proof}
Now for general supersingular $p|a_p$, we can evaluate $L_p^{\sharp}(E,\eta,X)$ and $L_p^{\flat}(E,\eta,X)$ at $X=0$:
\[
\begin{tabular}{|c||c|c|}
\hline
&$L_p^{\sharp}(E,\eta,0)$ & $L_p^{\flat}(E,\eta,0)$\\\hline\hline
$p$ odd, $\eta=1$&$(-a_p^2+2a_p+p-1)\frac{L(E,1)}{\Omega_E^+}$&$(2-a_p)\frac{L(E,1)}{\Omega_E^+}$\\\hline
$p$ odd, $\eta\neq 1$&$-pa_p\frac{L(E,\overline{\eta},1)}{\tau(\overline{\eta})\Omega_E^{\eta(-1)}}$&$-p\frac{L(E,\overline{\eta},1)}{\tau(\overline{\eta})\Omega_E^{\eta(-1)}}$\\\hline
$p=2$, $\eta=1$&$(-a_p^3+2a_p^2+2pa_p-a_p-2p)\frac{L(E,1)}{\Omega_E^+}$&$(-a_p^2+2a_p+p-1)\frac{L(E,1)}{\Omega_E^+}$\\\hline
$p=2$, $\eta\neq 1$&$-p^2a_p\frac{L(E,\overline{\eta},1)}{\tau(\overline{\eta})\Omega_E^{\eta(-1)}}$&$-p^2\frac{L(E,\overline{\eta},1)}{\tau(\overline{\eta})\Omega_E^{\eta(-1)}}$\\\hline
\end{tabular}
\]

%For $a_p\neq0$, $L_p^{\sharp}(E,\eta,0)$ is thus always a nonzero multiple of $L(E,\overline{\eta},1)$. The same holds for $L_p^{\flat}(E,\eta,0)$, with just one exception ($\eta=1$ and $a_p=3$). It thus seems reasonable to guess that:
\begin{propn}\label{atleastoneisnonzero}Let $\eta:\Delta\rightarrow \Z_p^\times$ be any character. Then at least one of $L_p^{\sharp}(E,\eta,X)$ and $L_p^{\flat}(E,\eta,X)$ is a nonzero function. If $L(E,\overline{\eta},1)\neq 0$, then they are both nonzero.
\end{propn}
\begin{proof}%We prove that $L_p^{\flat}(E,\eta,X)\neq0$ for odd $p$. The other cases can be proved similarly. 
By a theorem of Rohrlich (\cite{rohrlich}), $L_p(E,\alpha,\eta,\zeta_{p^n}-1)\neq0$ for $n \gg 0$. Thus, the vector $(L_p^{\sharp}(E,\eta,X),L_p^{\flat}(E,\eta,X))$ can't be zero. Further, if $L(E,\overline{\eta},1)\neq 0$, then the assertion follows from the table above.
\end{proof}

In view of Mazur's and Swinnerton-Dyer's conjecture in the ordinary case \cite[Conjecture 1]{mazurSD} and the fact that neither $L_p^{\sharp}(E,\eta,X)$ nor $L_p^{\flat}(E,\eta,X)$ vanish when $a_p=0$ (cf. \cite[Corollary 5.11]{pollacksthesis}), it seems reasonable to conjecture the following:
\begin{conjecture}\label{nonzeroness}Let $E$ be an elliptic curve, $p$ be a prime of good supersingular reduction, and $\eta:\Delta\rightarrow \Z_p^\times$ be any character. Then $L_p^{\sharp}(E,\eta,X)$ and $L_p^{\flat}(E,\eta,X)$ are both non-zero functions.
\end{conjecture}

Proposition \ref{atleastoneisnonzero} gives even more evidence for this conjecture via the fact that a positive proportion of elliptic curves has rank zero, cf. \cite{bhargavaarul}.

\begin{remark}\label{comparison}We can recover Kobayashi's and Pollack's results concerning $a_p=0$. We have \[\LL(X)=\links \log_p^+(1+X) & 0 \\ 0 & p\log_p^-(1+X) \rechts\text{for odd $p$,}\]  \[\LL(X)=\links 0 & -\log_p^+(1+X) \\ \log_p^-(1+X) & 0 \rechts\text{for $p=2$}.\]
Thus we find that
\[\hidari\log_\alpha^{\sharp}(1+X)\\ \log_\alpha^{\flat}(1+X)\migi=  \hidari \log_p^+(1+X)\\ \alpha\log_p^-(1+X)\migi \text{ if $p$ is odd,}\]\[
\hidari\log_\alpha^{\sharp}(1+X)\\ \log_\alpha^{\flat}(1+X)\migi=  \hidari -\frac{1}{2}\log_2^+(1+X)\alpha\\ \log_2^-(1+X)\migi \text{ if $p=2$.}\\
\]
This shows that the extra factor of $\frac{1}{2}$ in \cite[Theorem 5.6]{pollacksthesis} comes from an extra factor of $A^{-1}$ in Definition \ref{defofLL} of $\LL$ because the conductor $N$ is one larger than usual if $p=2$.
\[\begin{array}{ll}(L_2^{\sharp},L_2^{\flat})&=(-L_2^-,L_2^+) \text{ in Pollack's notation,}\\
(L_p^{\sharp},L_p^{\flat})&=\begin{cases}(L_p^+,L_p^-) \text{ in Pollack's notation}\\
(L_p^-,L_p^+)\text{ in Kobayashi's notation}\end{cases}\text{ if } p \text{ is odd.}\end{array}\]

The matrix $\tilde{Y}$ was inserted so as to match Kobayashi's construction for $a_p=0$. He constructed his $L_p^\pm$ from elements $(c_n^-,c_n^+)=\left(\delta_n^{-N+1},\delta_n^{-N}\right)$. This is why $c_0^+=-c_0$ rather than $c_0^+=c_0$ in \cite{kobayashisthesis}. Let $\widetilde{\omega_n}^-:=\prod_{m\in\N}\Phi_{2m-1}(1+X),  \widetilde{\omega_n}^+:=\prod_{m\in\N}\Phi_{2m}(1+X)$. Denoting the projection $\Lambda\rightarrow \Lambda_n$ by an overline,
\[\begin{array}{ccc} (P_n^-\varepsilon_\eta,P_n^+\varepsilon_\eta)&=(\overline{L_p^-(E,\eta,X)},\overline{L_p^+(E,\eta,X)}){\links \widetilde{\omega_n}^+ & 0 \\ 0 & \widetilde{\omega_n}^- \rechts}&=(\overline{L_p^-(E,\eta,X)},\overline{L_p^+(E,\eta,X)})\LL_nY_{-N}\\ \vspace{2mm}\parallel
\\
(P_n^{-N+1}\varepsilon_\eta,P_n^{-N}\varepsilon_\eta)&=\varepsilon_{\eta}\col_n(\mathbf{z}^{\eta(-1)})\LL_nY_{-N}&=(\overline{L_p^{\sharp}(E,\eta,X)},\overline{L_p^{\flat}(E,\eta,X)})\LL_nY_{-N}.
\end{array}\]
\end{remark}
\end{section}

\begin{section}{The Two Selmer Groups}\label{mainconjecture}
From now on, assume $p$ is odd.
\subsection{The Image of the Coleman Map}

Recall that the map $\col_n$ from Proposition \ref{col} was defined by

\[\xymatrix {H^1(k_n,T)\ar[r]^{\hspace{2.5mm}\col_n}\ar@/_8mm/[rd]^{P_n^{i+1}, P_n^i}& \frac{\Lambda_n \oplus \Lambda_n}{\ker \llll_n}\ar@{^(->}[d]^{\llll_n^i}\\
&\Lambda_{n}\oplus\Lambda_{n}.}
\]
We can split the limit of these maps $\col=\displaystyle\varprojlim_n \col_n$ and define $\sharp/\flat$-Coleman maps: \begin{definition}
$\col =: (\col^{\sharp}, \col^{\flat})$.
\end{definition}
However, we cannot split $\col_n$ in general, since $\ker \llll_n$ lives in $\Lambda_n\oplus \Lambda_n$. If $n=0$, we have $\ker \llll_0=0$ and thus we can put:
\begin{definition}$\col_0=:(\col_0^{\sharp},\col_0^{\flat})=(-a_pP_0^1+P_0^0,-P_0^1)$.
\end{definition}
Thanks\footnote{It would have been more natural to choose $-\tilde{Y}=(AB_{-})^{-1}$, but then we would have been off by a factor of $-1$ from Kobayashi's sign convention (cf. Remark \ref{comparison}) in the case $a_p=0$.} to $\tilde{Y}$, $\ker \llll_1=0\oplus X\Lambda_1$, so we can write $\col_1=(\col_1^{\sharp}, \col_1^{\flat})$ and $\col_1^{\sharp}=-P_1^1$.
\begin{propn}\label{colupsilon} $\col^{\flat}$ is surjective.
\end{propn}
\begin{lemma}$\F(\gm_0)$ is generated by $(c_0^\sigma)_{\sigma \in \G_0}$ as a $\Z_p$-module.
\end{lemma}
\begin{proof} We have $\Tr_{k_0/\Q_p}(c_0)=(a_p-2) c_{-1}$ by Theorem \ref{tracerelations} (2), and we know that $\F(\gm_0)$ is generated by $(c_0^\sigma)$ and $(c_{-1}^\sigma)$.
\end{proof}
\begin{lemma}\label{generators}$(c_0^\sigma)_{\sigma \in \G_0}$ are a basis for $\F(\gm_0)$ as a $\Z_p$-module.
\end{lemma}

\begin{proof}$\rank_{\Z_p}(\F(\gm_0)^\eta)=1$ for any character $\eta: \Delta \rightarrow \Z_p^\times$.
\end{proof}
\renewcommand{\qedsymbol}{{$\mathit{QED}$}}
\begin{proof} [Proof of Proposition \ref{colupsilon}] We can prove the surjectivity of $\col_0^{\flat}=-P_0^1$ as in \cite[Proposition 8.23]{kobayashisthesis}: $P_0^1(z)=\displaystyle \sum_{\sigma\in\Delta}(c_0^\sigma,z)_0\sigma$, so from Lemma \ref{generators}, $P_0^1$ is given by

\[H^1(k_0,T)\rightarrow \Hom(\F(\gm_0),\Z_p)\cong \Lambda_0,\]
where the first map comes from the pairing given by the cup product $(\:,\:)_0:\F(\gm_0)\times H^1(k_0,T)\rightarrow \Z_p$ (see \cite[(8.23) on page 18]{kobayashisthesis}), and the last identification is $f\mapsto \displaystyle \sum_{\sigma \in \Delta} f(c_0^\sigma)\sigma$.

Now $H^1(k_0,T)\rightarrow \Lambda_0$ is surjective since its Pontryagin dual is the injection \[{\F(\gm_0)\otimes \Q_p/\Z_p}={\hat{E}(k_0)\otimes\Q_p/\Z_p}\hookrightarrow H^1(k_0,V/T),\] where $V=T\otimes\Q_p$. In fact, $E(k_0)\otimes\Q_p/\Z_p\hookrightarrow H^1(k_0,V/T)$ is always an injection, so we want $\hat{E}(k_0)\otimes\Q_p/\Z_p\rightarrow E(k_0)\otimes\Q_p/\Z_p$ to be one as well. Now this follows from the cokernel $\tilde{E}(\mathbb{F}_p)$ of $\hat{E}(k_0)\hookrightarrow E(k_0)$ having no $p$-torsion, since $p$ is supersingular.

Consider the following commutative diagram with exact rows:

\[\xymatrix {\displaystyle\varprojlim_n H^1(k_{n},T)\ar@{->}[r]^{\hspace{10mm}\col^{\flat}}\ar[d]^{\cor}& \Lambda\ar[d]\ar[r] &\coker(\col^{\flat})\ar[r]\ar[d]&0\\
H^1(k_0,T)\ar[r]^{\hspace{0mm}\col_0^{\flat}}& \Lambda_0=\Z_p[\Delta]\ar[r]&\coker(\col_0^{\flat})\ar[r]&0}
\]
 The surjectivity of $\col^{\flat}$ follows from Nakayama's Lemma if we know that the corestriction $\cor$ is surjective, which is indeed the case since this is true at finite levels: $H^1(k_n,T)\rightarrow H^1(k_m,T)$ is surjective if $n\geq m$. This can be seen as follows: The kernel of the restriction map $H^1(k_m,V/T)\rightarrow H^1(k_n,V/T)$ is $H^1(\Gal(k_n/k_m), H^0(k_n,V/T))$ by the inflation-restriction sequence. But the formal group $\F$ has no $p$-power torsion in $k_n$ by Lemma \ref{noptorsion}, so $H^0(k_n,V/T)=0$. Thus the kernel is zero, so we are done after using Tate duality.
\end{proof}

\begin{propn}\label{proptheta} If $\eta$ is trivial, then $\varepsilon_\eta \col^{\sharp}$ is surjective. If $\eta$ is nontrivial, then $\im(\varepsilon_\eta \col^{\sharp})=J^\eta$. Here, ${J= (a_p+a_pX+\frac{a_p}{p}X^2,X)\subset \Lambda}$.
\end{propn}

\begin{lemma}\label{decomposition} We have an exact sequence
\[0 \longrightarrow \F(\gm_{-1})
\stackrel{\Delta}\longrightarrow \F(\gm_0) \oplus \mathcal{C}_{ss}^{\delta_1^{-1}}(\gm_1)
\stackrel{[-]}{\longrightarrow} \F(\gm_1)
\longrightarrow 0,\]
where  $\mathcal{C}_{ss}^{\delta_1^{-1}}(\gm_1)$ is the $\Z_p$-submodule of $\F(\gm_1)$ generated by $\{(\delta_1^{-1})^{\sigma}\}_{\sigma\in\G_1}$, $\Delta$ the diagonal map, and $[-]: (a,b)\mapsto (a-b).$
\end{lemma}
\begin{proof}
$\{(\delta_1^{-1})^{\sigma}\}_{\sigma\in\G_1}$ and $\{(\delta_1^0)^\sigma\}_{\sigma\in\G_1}=\{c_0^\sigma\}_{\sigma\in\G_1}$ generate $\F(\gm_1)$ as a $\Z_p$-module, so we have to prove that $\mathcal{C}_{ss}^{\delta_1^{-1}}(\gm_1) \cap \F(\gm_0)= \F(\gm_{-1})$. Let $P \in \mathcal{C}_{ss}^{\delta_1^{-1}}(\gm_1) \cap \F(\gm_0)$. Then $\Tr_{1/0}P \in \F(\gm_{-1})$, since this holds for all elements of $\mathcal{C}_{ss}^{\delta_1^{-1}}$. But since we also have $P \in \F(\gm_0)$, $\Tr_{1/0}P=pP$. Thus, $p(P^\sigma-P)=0$ for all $\sigma \in \Gal(k_0/\Q_p)$. But $\F(\gm_0)$ has no $p$-torsion, so $P$ is invariant under $\Gal(k_0/\Q_p)$, whence $P \in \F(\gm_{-1})$. Also, $\F(\gm_{-1})\subset \mathcal{C}_{ss}^{\delta_1^{-1}}(\gm_1)$ since $\Tr_{1/0}\delta_1^{-1}=\delta_0^0=c_{-1}$ and $\{c_{-1}^\sigma\}_{\sigma \in \Delta}$ generate $\F(\gm_{-1})$ as a $\Z_p$-module (see e.g. \cite{pollack}).
\end{proof}

\begin{corollary}\label{rank} If $\eta:\Delta\rightarrow \Z_p^{\times}$ is the trivial character, then $\rank_{\Z_p} (\mathcal{C}_{ss}^{\delta_1^{-1}}(\gm_1))^\eta=p$.
\end{corollary}
\begin{proof} $\rank_{\Z_p}\F(\gm_{-1})^\eta=1$, $\rank_{\Z_p}\F(\gm_0)^\eta=1$, and $\rank_{\Z_p}\F(\gm_{1})^\eta=p$.
\end{proof}

\renewcommand{\qedsymbol}{{$\textit{QED}$}}
\begin{proof}[Proof of Proposition \ref{proptheta}] The idea is to use the map $\col_1^{\sharp}$ in the way we used $\col_0^{\flat}$ in the proof of Proposition \ref{colupsilon}. For $\eta$ trivial, $(\varepsilon_\eta {\delta_1^{-1}}^\sigma)_{\sigma\in \Gamma_1}$ is a basis for $\mathcal{C}_{ss}^{\delta_1^{-1}}(\gm_1)^\eta$ by the above Corollary \ref{rank}. Here, we have written $\G_1=\Delta \times \Gamma_1$, and the proof can now proceed as in the $\col^{\flat}$-case, since from Lemma \ref{decomposition}, $\mathcal{C}_{ss}^{\delta_1^{-1}}(\gm_1)^\eta\cong \F(\gm_1)^\eta$, so that $\mathcal{C}_{ss}^{\delta_1^{-1}}(\gm_1)^\eta\otimes\Q_p/\Z_p\hookrightarrow H^1(k_1,V/T)$ is an injection.

For $\eta$ non-trivial, we prove that $\im(\varepsilon_\eta \col_1^{\sharp})=\varepsilon_\eta(a_p+a_pX+\frac{a_p}{p}X^2,X)\Lambda_1$. Then we use Nakayama's Lemma as above.

We have that $\col_1^{\sharp}=-P_1^1=-\frac{a_p}{p}P_1^0+P_1^{-1}$, and from the previous Lemma \ref{decomposition} that $\F(\gm_1)^{\eta^{-1}}\cong \F(\gm_0)^{\eta^{-1}}\oplus \mathcal{C}_{ss}^{\delta_1^{-1}}(\gm_1)^{\eta^{-1}}$. Thus we can describe $-\varepsilon_\eta \col_1^{\sharp}$ by
\[\xymatrix {H^1(k_1,T)^\eta\ar[r]& \Hom(\F(\gm_1)^{\eta^{-1}},\Z_p)\ar[r]\ar@{=}[d]^{\,Lemma \,\ref{decomposition}}&\Lambda_1^\eta\\
&\Hom(\F(\gm_0)^{\eta^{-1}}\oplus \mathcal{C}^{\delta_1^{-1}}(\gm_1)^{\eta^{-1}},\Z_p).\ar[ru]}
\]
The first (surjective) map is induced by the pairing in Definition \ref{kuriharanoPn}, and the second map is \[f\mapsto \varepsilon_\eta\displaystyle \sum_{\sigma\in\G_1}f(\varepsilon_{\eta^{-1}}{\delta_1^1}^\sigma)\sigma=|\Delta|\varepsilon_{\eta}\displaystyle \sum_{\sigma\in\Gamma_1}f(\varepsilon_{\eta^{-1}}{\delta_1^1}^\sigma)\sigma\in\Z_p[\Delta][\Gamma_1]^\eta=\Lambda_1^\eta.\] From $\delta_1^1=\frac{a_p}{p}\delta_1^0-\delta_1^{-1}$, 
\[\varepsilon_\eta\sum_{\sigma\in\Gamma_1}f(\varepsilon_{\eta^{-1}}{\delta_1^1}^\sigma)\sigma=\varepsilon_\eta\sum_{\sigma\in\Gamma_1}\left(\frac{a_p}{p}f(\varepsilon_{\eta^{-1}}\delta_1^0)-f(\varepsilon_{\eta^{-1}}{\delta_1^{-1}}^\sigma)\right)\sigma.\]
$\varepsilon_{\eta^{-1}}\delta_1^0=\varepsilon_{\eta^{-1}} c_0$ and $(\varepsilon_\eta\delta_1^{-1})^\sigma_{\sigma\in \Gamma_1}$ are each a basis for $\F(\gm_0)^{\eta^{-1}}$ and generators for $\mathcal{C}_{ss}^{\delta_1^{-1}}(\gm_1)^{\eta^{-1}}$. From Lemma \ref{decomposition},
\[\begin{array}{ll}
\rank_{\Z_p}\mathcal{C}_{ss}^{\delta_1^{-1}}(\gm_1)^{\eta^{-1}}&=\rank_{\Z_p}\F(\gm_1)^{\eta^{-1}}+\rank_{\Z_p}\F(\gm_{-1})^{\eta^{-1}}-\rank_{\Z_p}\F(\gm_0)^{\eta^{-1}}\\
&=p-1,\text{ since $\rank_{\Z_p}(\F(\gm_{-1}^{\eta^{-1}})=0$ for ${\eta^{-1}}$ nontrivial.}
\end{array}\]
But note that $\displaystyle \sum_{\sigma\in\Gamma_1}(\varepsilon_{\eta^{-1}}\delta_1^{-1})^\sigma\sigma=\varepsilon_{\eta^{-1}} \Tr_{1/0}\delta_1^{-1}=\varepsilon_{\eta^{-1}}\delta_0^0=0$.
Thus, $(\varepsilon_{\eta^{-1}}{\delta_1^{-1}}^\sigma)_{\sigma\in\Gamma_1-\{1\}}$ is a basis for $\mathcal{C}_{ss}^{\delta_1^{-1}}(\gm_1)^{\eta^{-1}}$, so
\[\Hom(\mathcal{C}_{ss}^{\delta_1^{-1}}(\gm_1)^{\eta^{-1}},\Z_p)\cong\varepsilon_\eta X\Lambda_1\]
by $f \mapsto \varepsilon_\eta\sum_{\sigma\in\Gamma_1} f(\varepsilon_{\eta^{-1}}{\delta_1^{-1}}^\sigma)\sigma$. Note that $X\Lambda_1=\ker(\Z_p[\Delta][\Gamma_1]\rightarrow\Z_p[\Delta]).$

Since $|\Delta|=p-1$ is a unit and the dual basis of the basis $\{\varepsilon_{\eta}^{-1}(\delta_1^{-1})^\sigma\}_{\sigma\in\Gamma_1-\{1\}}$  for $\mathcal{C}_{ss}^{\delta_1^{-1}}(\gm_1)^{\eta^{-1}}$ has image $\{\varepsilon_{\eta}(\sigma-1)\}_{\sigma \in\Gamma_1-\{1\}}$ in $\varepsilon_{\eta}\Lambda_1$,
\[\im(\varepsilon_\eta \col_1^{\sharp})=\varepsilon_\eta\left(\frac{a_p}{p}\Z_p\sum_{\sigma\in\Gamma_1}\sigma-\sum_{\sigma\in \Gamma-\{1\}}\Z_p(\sigma-1)\right).\]
For $p\geq5, a_p=0$, so $\im(\varepsilon_\eta \col_1^{\sharp})=\varepsilon_\eta\sum_{\sigma\in\Gamma-\{1\}}\Z_p(\sigma-1)=\varepsilon_\eta X\Lambda_1$.%since ${\delta_1^1}^\sigma$ is a basis for $\mathcal{C}^{\delta_1^1}(\gm_1),\Hom(\mathcal{C}^{\delta_1^1}(\gm_1)^\eta,\Z_p) \cong \{ \varepsilon_\eta \sum_{\sigma_\in \Gamma_1} f(\varepsilon_\eta {\delta_1^1}^{\sigma}) \sigma, f\in \Hom(\mathcal{C}^{\delta_1^1}(\gm_1^\eta), \Z_p)\}$.

%We have $\varepsilon_\eta \sum_{\sigma\in\Gamma_1} f(\varepsilon_\eta {\delta_1^1}^\sigma)^\sigma= \varepsilon_\eta \sum_{\sigma \in \Gamma_1} \frac{a_p}{p} f(\varepsilon_\eta \delta_1^0) \sigma - f (\varepsilon_\eta {\delta_1^{-1}}^\sigma)\sigma$.

%But $\varepsilon_\eta \delta_1^0$ and $\varepsilon_\eta {\delta_1^{-1}}^\sigma$ are a basis for $\mathcal{C}^{\delta_1^1}(\gm_1)^\eta,$ so 

%$\im(\col_1^{\sharp})= \varepsilon_\eta (\frac{a_p}{p} \Z_p \sum_{\sigma \in \Gamma_1}\sigma - \sum_{\sigma \in \Gamma_1}\Z_p \sigma-A)$, where $A=s[1]$, where $1$ is the identity in the Galois group, and $s$ is the sum of the coefficients in front of the $\sigma$s in the sum to the left.
 
If $p=3$, then \[\im(\varepsilon_\eta \col_1^{\sharp})= \varepsilon_\eta((\frac{a_p}{p}+\frac{a_p}{p}(1+X)+\frac{a_p}{p}(1+X)^2)\Lambda_1+X\Lambda_1),\] since

 \[\Z_p(\frac{a_p}{p}+\frac{a_p}{p}(1+X)+\frac{a_p}{p}(1+X)^2) \cong \Lambda_1(\frac{a_p}{p}+\frac{a_p}{p}(1+X)+\frac{a_p}{p}(1+X)^2).\]
 
 Thus, 
 $\im(\varepsilon_\eta \col_1^{\sharp})= \varepsilon_\eta(a_p+a_pX+\frac{a_p}{p}X^2,X)\Lambda_1$. Now we can apply the arguments as in the proof of Proposition \ref{colupsilon} to show $\im(\varepsilon_\eta \col^{\sharp})=J^{\eta}$. Note that $\omega_1(X)=(1+X)^p-1\in p\Lambda+X\Lambda$.
\end{proof}

\subsection{The Main Conjecture}
We keep the notation of \cite{kobayashisthesis}: $K_n= \Q(\zeta_{p^{n+1}}), K_{-1}=\Q,$ and $K_\infty= \displaystyle \bigcup_n K_n$. 
Let $\gp_n$ be the place in $K_n$ and $\gp$ the place in $K_{\infty}$ above $p$. Kobayashi constructed the two $\pm$-Selmer groups
\[ \Sel^\pm(E/K_n)= \ker \left(\Sel(E/K_n)\longrightarrow \frac{E(K_{n,\gp_n})\otimes \Q_p/\Z_p}{E^\pm(K_{n,\gp_n})\otimes \Q_p/\Z_p}\right),\]
and put
\[\Sel^\pm(E/K_\infty)=\varinjlim_n\Sel^\pm(E/K_n).\]
Thus, we could also write
\[\Sel^\pm(E/K_\infty)=\ker\left(\Sel(E/K_\infty)\longrightarrow \frac{E(K_{\infty, \gp})\otimes \Q_p/\Z_p}{E^\pm(K_{\infty, \gp})\otimes \Q_p/\Z_p}\right),\]
where
\[\Sel(E/K_\infty)=\displaystyle \varinjlim_n \Sel(E/K_n), {E(K_{\infty, \gp})\otimes \Q_p/\Z_p}=\varinjlim_n {E(K_{n,\gp_n})\otimes \Q_p/\Z_p},\]
and 
\[{{E^\pm(K_{\infty,\gp})\otimes \Q_p/\Z_p}=\displaystyle \varinjlim_n {E^\pm(K_{n,\gp_n})\otimes \Q_p/\Z_p}}.\]

 The essential property is that $E^\pm(K_{\infty,\gp})\otimes\Q_p/\Z_p$ is the exact annihilator of $\ker(\col^\pm)$ under the local Tate pairing

\[ {\displaystyle\varprojlim_nH^1(K_{n,\gp_n},T)\times \varinjlim_n H^1(K_{n,\gp_n},V/T)\rightarrow \Q_p/\Z_p}.\]
Thus, we make the following definition:
\begin{definition} We let $E^{\sharp}_{\infty,\gp}\, (\text{resp. }E^{\flat}_{\infty,\gp})$ be the exact annihilator of $\ker \col^{\sharp}\, (\text{resp. }\ker \col^{\flat})$ under the local Tate pairing.
\end{definition}
\begin{lemma} We have $E(K_{\infty,\gp})\otimes \Q_p/\Z_p=\displaystyle\varinjlim_n H^1(K_{n,\gp_n},V/T)$.\end{lemma}
\begin{proof}$E(K_{n,\gp_n})\otimes \Z_p$ and $E(K_{n,\gp_n})\otimes  \Q_p/\Z_p$ are exact annihilators of each other, and we have $\varprojlim_n E(K_{n,\gp_n})\, \hat\otimes\, \Z_p=0$ since $p$ is supersingular (by \cite{coatesgreenberg} or \cite{hazewinkel}).\end{proof}
\begin{definition}We now define our two Selmer groups:
\[\Sel^{\sharp}(E/K_\infty):=\ker\left(\Sel(E/K_\infty)\longrightarrow \frac{E(K_{\infty,\gp})\otimes\Q_p/\Z_p}{E^{\sharp}_{\infty,\gp}}\right),\]
\[\Sel^{\flat}(E/K_\infty):=\ker\left(\Sel(E/K_\infty)\longrightarrow \frac{E(K_{\infty,\gp})\otimes\Q_p/\Z_p}{E^{\flat}_{\infty,\gp}}\right).\]
We also define their Pontryagin duals $\X^{*}(E/K_\infty):=\Hom(\Sel^{*}(E/K_\infty),\Q_p/\Z_p)$ for $*\in\{\sharp,\flat\}$.
\end{definition}

%\begin{theorem}Let $*=\sharp$ or $*=\flat$. Then $\X^*(E,K_\infty)$ is $\Lambda$-torsion.
%\end{theorem}
%\begin{proof}Since $L_p^*(E,\eta,X)\neq0$ for both $*=\sharp$ and $*=\flat$ (Proposition \ref{nonzeroness}), we can use the same arguments as in \cite{kobayashisthesis}.\end{proof}

\begin{definition}Let $j$ the natural morphism  $\Spec K_n \rightarrow \Spec O_{K_n}[\frac{1}{p}]$, and let $H^q(\text{Spec }O_{K_n}[\frac{1}{p}],j_*T)$ be the $q$-th \'etale cohomology group. We define $\mathbf{H}^1(T):=\displaystyle \varprojlim_n$ $ H^q(\text{Spec }O_{K_n}[\frac{1}{p}],j_*T)$. We let $\mathbf{Z}(T)$ be the $\Lambda$-submodule of $\mathbf{H}^1(T)$ generated by Kato's zeta elements $\mathbf{z}^+$ and $\mathbf{z}^-$. (We can make this definition since $p$ is supersingular. See \cite[Section 5]{kobayashisthesis} or \cite{kato} for more details). 
\end{definition}
\begin{definition}$\X^0(E/K_\infty)$ is the Pontryagin dual to Kurihara's Selmer group (see \cite{kurihara}), which is 
\[\Sel^0(E/K_\infty):=\ker\left(\Sel(E/K_\infty)\longrightarrow E(K_{\infty,\gp})\otimes\Q_p/\Z_p\right).\]
\end{definition}

\begin{theorem}\label{cotorsion}Let $p$ be an odd supersingular prime, $\eta:\Delta\rightarrow \Z_p^\times$ be a character, and $*\in\{\sharp,\flat\}$ be chosen so that $L_p^*(E,\eta,X)\neq 0$. Then $\X^*(E/K_\infty)^\eta$ is $\Z_p[[X]]$-torsion.
\end{theorem}

\begin{proof} From \cite[Proposition 7.1 and the limit of (7.18)]{kobayashisthesis}, we have the exact sequence
\[\mathbf{H}^1(T)^\eta\rightarrow \mathbf{H}^1_{\text{Iw}}(T)^\eta\rightarrow \X(E/K_\infty)^\eta\rightarrow \X^0(E/K_\infty)^\eta\rightarrow 0,\]
noting that  $\varprojlim_n E(K_{n,\gp_n})\hat\otimes\Z_p=0$.
Since $L_p^*(E,\eta,X)\neq 0$ by assumption, the arguments in the proof of \cite[Theorem 7.3 i]{kobayashisthesis} provide us with an injection $\iota$ in the exact sequence of $\Z_p[[X]]$-modules
\begin{equation}\label{exactsequence}0\rightarrow \mathbf{H}^1(T)^\eta\stackrel{\iota}{\longrightarrow} (\mathbf{H}^1_{\text{Iw}}(T)^\eta/\ker \varepsilon_\eta\col^*)\rightarrow \X^*(E/K_\infty)^\eta\rightarrow \X^0(E/K_\infty)^\eta\rightarrow 0.\end{equation}
To see that the entire sequence is exact, we can use the Cassels-Poitou-Tate exact sequence (cf. \cite[Appendix A.3.2]{perrinriou}  or \cite[Theorem 1.5]{coatessujatha}) and the discussion in \cite{kobayashisthesis} preceding $(7.16)$. But $\coker(\iota)$ is killed by $L_p^*(E,\eta,X)\neq 0$ and is thus $\Z_p[[X]]$-torsion. Now $\X^0(E/K_\infty)^\eta$ is a torsion $\Z_p[[X]]$-module (\cite[Corollary 7.2]{kobayashisthesis}), so we are done.
\end{proof}

A consequence of Conjecture \ref{nonzeroness} would then be the following conjecture.
\begin{conjecture} Both $\X^\sharp(E/K_\infty)$ and $\X^\flat(E/K_\infty)$ are $\Lambda$-torsion.
\end{conjecture}
Using work of Kato in an analogous way to Kobayashi, we now prove the following theorem:

\begin{theorem}\label{formaltheorem}
Let $p$ be an odd supersingular prime, and $*\in\{\sharp,\flat\}$ so that $L_p^*(E,\eta,X)\neq 0$. Then for some integer $n\geq 0,$
\[\Char(\X^*\left(E/K_\infty)^\eta\right)\supseteq \left(p^n\frac{1}{X}L_p^*(E,\eta,X)\right) \text{if }*=\sharp, \eta \neq 1, \text{and } a_p=0,\]
\[\Char(\X^*\left(E/K_\infty)^\eta\right)\supseteq \left(p^n L_p^*(E,\eta,X)\right)\text{ for all other cases.}\]
%\[\Char(\X^{\flat}\left(E/K_\infty)^\eta\right)\supseteq \left(\frac{p^n}{X}L_p^{\sharp}(E,\eta,X), \frac{p^nL_p^{\sharp}(E,\eta,X)}{a_p(1+X+\frac{1}{p}X^2)}\right) \text{for }\eta \neq 1.\]
%(In the last term we ignore the second generator of the ideal when $a_p=0$.)
%\[\Char(\X^{\flat}\left(E/K_\infty)^\eta\right)\supseteq p^n\Char\left(J^\eta/L_p^{\sharp}(E,\eta,X)\right) \text{for }\eta \neq 1.\]

 Further, if the $p$-adic representation $\Gal(\overline{\Q}/\Q)\rightarrow \GL_{\Z_p}(T)$ on the automorphism group of the $p$-adic Tate module $T$ is surjective, we can take $n=0$.
\end{theorem}

%To do this, first not that the arguments of \cite[Section 7]{kobayashisthesis} apply to our case and we can extend \cite[Theorem 7.4]{kobayashisthesis} to :
To prove this, we need a proposition by Kurihara \cite[Proposition 7.1 ii]{kobayashisthesis} and a theorem by Kato \cite[Theorem 12.5]{kato}:
\begin{propn}[Kurihara] $\mathbf{H}^2(T)$ and $\X^0(E/K_\infty)$ are isomorphic as $\Lambda$-modules.
\end{propn}

\begin{theorem}[Kato] Suppose that $\Gal(\overline{\Q}/\Q)\rightarrow \GL_{\Z_p}(T)$ is surjective. Then 
\[\Char\left(\mathbf{H}^2(T)^\eta\right)\supseteq \Char\left(\mathbf{H}^1(T)^\eta/\mathbf{Z}(T)^\eta\right).\]
\end{theorem}

Theorem \ref{formaltheorem} then follows from the following exact sequences: 

\begin{propn}\label{threeexactsequences} Choose $*\in\{\sharp,\flat\}$ so that $L_p^*(E,\eta,X)\neq 0$. Then there are exact sequences
\[0\rightarrow \mathbf{H}^1(T)^\eta/\mathbf{Z}(T)^\eta\rightarrow J^\eta/(L_p^*(E,\eta,X))\rightarrow \X^*(E/K_\infty)^\eta\rightarrow\X^0(E/K_\infty)^\eta\rightarrow 0 \text{ if } *=\sharp, \eta \neq 1,\]
\[0\rightarrow \mathbf{H}^1(T)^\eta/\mathbf{Z}(T)^\eta\rightarrow \Lambda^\eta/(L_p^*(E,\eta,X))\rightarrow \X^*(E/K_\infty)^\eta\rightarrow\X^0(E/K_\infty)^\eta\rightarrow 0\text{ if not.}\]
%\[0\rightarrow \mathbf{H}^1(T)^\Delta/\mathbf{Z}(T)^\Delta\rightarrow \Lambda^\Delta/(L_p^{\sharp}(E,\Delta,X))\rightarrow \X^{\sharp}(E/K_\infty)^\Delta\rightarrow\X^0(E/K_\infty)^\Delta\rightarrow 0\]
\end{propn}

\begin{proof}This follows from the exact sequence (\ref{exactsequence}) in the proof of Theorem \ref{cotorsion}, Proposition \ref{colupsilon}, and Proposition \ref{proptheta}.
\end{proof}

Since $p$ is odd supersingular, Kato's main conjecture \cite[Conjecture 12.10]{kato} reads:

\begin{conjecture}[Kato] 
For a character $\eta:\Delta\rightarrow \Z_p^\times, \Char \left(\X^0(E/K_\infty)^\eta\right)=\Char\left(\mathbf{H}^1(T)^\eta/\mathbf{Z}(T)^\eta\right).$
\end{conjecture}

%\begin{propn} Let $\eta$ be nontrivial. Then $L_p^{\sharp}(E,\eta,0)=0$ and if $a_p \neq 0$, $L_p^{\sharp} (E,\eta, \zeta_3)=0$, where $\zeta_3$ is a primitive third root of unity. Thus, $X|L_p^{\sharp}(E,\eta,X)$ and for $a_p\neq 0, a_p(1+X+\frac{1}{p}X^2)|L_p^{\sharp}(E,\eta,X)$.
%\end{propn}
%\begin{proof}Calcuation. (???????$B!k(B???j
%\end{proof}

In view of Proposition \ref{threeexactsequences}, the following conjecture is equivalent to the main conjecture of Kato, and to that of Perrin-Riou \cite{perrinriou} :

\begin{mainconjecture} Let $p$ be an odd supersingular prime, and choose $*\in\{\sharp,\flat\}$ so that $L_p^*(E,\eta,X)\neq 0$. In Theorem \ref{formaltheorem}, $n=0$ and the three inclusions are three equalities:
\[\Char(\X^*\left(E/K_\infty)^\eta\right)= \left(\frac{1}{X}L_p^*(E,\eta,X)\right) \text{ if }*=\sharp, \eta \neq 1, \text{and } a_p=0,\]
\[\Char(\X^*\left(E/K_\infty)^\eta\right)= \left(L_p^*(E,\eta,X)\right)\text{ for all other cases.}\]
%\[\Char(\X^{\flat}\left(E/K_\infty)^\eta\right)= \left(L_p^{\flat}(E,\eta,X)\right),\]
%\[\Char(\X^{\sharp}\left(E/K_\infty)^\eta\right)= \left(L_p^{\sharp}(E,\eta,X)\right) \text{ for trivial }\eta,\text{ and} \]
%\[\Char(\X^{\flat}\left(E/K_\infty)^\eta\right)= \left(\frac{L_p^{\sharp}(E,\eta,X)}{X}, \frac{L_p^{\sharp}(E,\eta,X)}{a_p(1+X+\frac{1}{p}X^2)}\right), \text{ for nontrivial }\eta.\]
%Here again, we ignore the second generator if $a_p=0$.
%\[\Char(\X^{\flat}\left(E/K_\infty)^\eta\right)=\left(L_p^{\sharp}(E,\eta,X)\right) \text{for nontrivial }\eta.\]
\end{mainconjecture}

Theorem \ref{formaltheorem} then follows from Kato's divisibility statement in \cite{kato}. Note that for $a_3\neq0$, $\Char(J^\eta/L_3^{\sharp}(E,\eta,X))=(L_3^{\sharp}(E,\eta,X))$. (If $a_3\neq0$ and $\eta$ nontrivial, $\Lambda^\eta/J^\eta\cong \mathbb{F}_3$ is pseudo-null.)

\begin{open problem} One of the insights in \cite{kobayashisthesis} was to describe the $E^\pm(K_{n,\gp})$ explicitly using trace maps. If one could explicitly write $E^{\sharp/\flat}_{\infty,\gp}=\displaystyle \varinjlim_n {E^{\sharp/\flat}(K_{n,\gp})\otimes\Q_p/\Z_p}$ for some easily defined $E^{\sharp/\flat}(K_{n,\gp}),$ this should lead to some interesting applications.
\end{open problem}

\begin{remark} We also remark that the techniques of Pollack and Rubin \cite{pollackrubin} can not be extended, since for an elliptic curve with complex multiplication supersingular at $p$, we always have $a_p=0$. 
\end{remark}

\begin{small}
\textit{Acknowledgments.} Our thanks go to everybody who deserves them, in particular Takeshi Tsuji, both for his initial bit of skepticism for this topic and relentless encouragement later, improvements of various proofs, and a very valuable suggestion that went into Lemma \ref{importantlemma}; to Shinichi Kobayashi, Robert Pollack, and Antonio Lei %and Jonathan Pottharst [IF ANALYTIC ARGUMENT WORKS OUT]
 for helpful comments and emails, and to the buoyant, cheerful atmosphere of room 406 at Todai Math. For helpful suggestions on the exposition and for pointing out mistakes, we thank Barry Mazur, Robert Pollack, Joseph Silverman, and the anonymous referee.
\end{small}
\end{section}

%% appendix Sections are then done as normal Sections


\end{document}